\documentclass{amsart}

\usepackage[normalem]{ulem}

\usepackage{lineno}

\usepackage{soul}

\usepackage{amssymb, amsmath, amsthm, amsfonts, euscript, enumerate, color}
\usepackage[hypertexnames=false,
    backref=page,
    pdftex,
    pdfpagemode=UseNone,
    breaklinks=true,
    extension=pdf,
    colorlinks=true,
    linkcolor=black,
    citecolor=black,
    urlcolor=blue,
]{hyperref}

\usepackage[dvipsnames]{xcolor}
\usepackage{tikz-cd}

\usepackage{amscd}

\usepackage{bbm}
\usepackage{color,soul}
\usepackage[shortlabels]{enumitem}
\usepackage[english]{babel}

\usepackage{tabularx}
\usepackage{array}

\usepackage{setspace}

\theoremstyle{plain}
\newtheorem{theorem}{Theorem}[section]
\newtheorem{lemma}[theorem]{Lemma}
\newtheorem{proposition}[theorem]{Proposition}

\newtheorem*{theorem*}{Theorem}

\theoremstyle{definition}
\newtheorem{definition}[theorem]{Definition}

\newtheorem*{example*}{Example}
\newtheorem*{definition*}{Definition}

\theoremstyle{remark}
\newtheorem{remark}[theorem]{Remark}

\newcommand{\Q}{\mathbb{Q}} 
\newcommand{\R}{\mathbb{R}} 
\newcommand{\C}{\mathbb{C}} 

\usepackage{graphicx} 
\usepackage{comment}

\title{p-Symplectic and p-pluriclosed structures on solvmanifolds}
\author{Ettore Lo Giudice and Adriano Tomassini}
\address[Ettore Lo Giudice, Adriano Tomassini]{
Dipartimento di Scienze Matematiche, Fisiche e Informatiche
Unit\`a di Matematica e Informatica\\
Universit\`a degli Studi di Parma\\
Parco Area delle Scienze 53/A, 43124\\
Parma, Italy}
\email{ettore.logiudice@unipr.it}
\email{adriano.tomassini@unipr.it}

\thanks{The first author is partially supported by GNSAGA of INdAM. The second author is partially supported by the project PRIN2022 “Real and Complex Manifolds: Geometry and Holomorphic Dynamics” (Project code: 2022AP8HZ9) and by GNSAGA of INdAM}

\begin{document}

\maketitle
\tableofcontents
\begin{abstract}
    Let $(M,J)$ be a $n$-dimensional complex manifold: a {\em $p$-K\"ahler structure} (resp. {\em $p$-symplectic structure}) on $M$ is a real, closed $(p,p)$-transverse form $\Omega$ (resp. real, closed $2p$-form whose $(p,p)$-component is transverse). We give obstructions to the existence of such structures on compact complex manifolds. We provide several families of compact complex manifolds which admit both $(n-1)$-symplectic structures and special Hermitian metrics. 
\end{abstract} 

\section{Introduction}
A {\em K\"ahler structure} on a $2n$-dimensional smooth manifold $M$ is given by a triple $(J,g,\omega)$, where $J$ is a complex structure on $M$, that is a $(1,1)$-tensor field satisfying $J^2=-\hbox{\rm Id}_{TM}$ whose torsion tensor vanishes, $g$ is a $J$-Hermitian metric, namely a Riemannian metric on $M$ such that $J$ acts on $M$ as an isometry and $\omega(\cdot,\cdot)=g(\cdot, J\cdot)$ is the fundamental form of $g$ which satisfies $d\omega=0$.\newline 
The existence of a K\"ahler structure on a compact manifold imposes strong constraints from the topological point of view, e.g., the odd index Betti numbers are even, the even index Betti numbers are positive, the de Rham complex of $M$ is a formal Differential Graded Algebra, that is $M$ is formal in the sense of Sullivan \cite{DGMS1975}. Furthermore, the complex manifold $(M,J)$ underlying to the K\"ahler manifold $(M,J,g,\omega)$ satisfies the $\partial\overline{\partial}$-Lemma and the symplectic manifold $(M,\omega)$ satisfies the Hard Lefschetz Condition. \newline
Nevertheless, many examples of compact complex manifolds with no K\"ahler structure can be naturally constructed as compact quotients of connected and simply connected solvable Lie groups by lattices. Such manifolds can be equipped with special Hermitian metrics, whose fundamental form satisfies mild differential conditions than the K\"ahler one.\newline
We briefly recall that if $(M,J)$ is a complex manifold of complex dimension $n$, then a metric $g$, with fundamental form $\omega$, is called 
\begin{itemize}
    \item {\em strong-K\"ahler with torsion} (shortly, {\em SKT}) if $\partial \overline{\partial} \omega = 0$ (see \cite{B1989});
    \item {\em Astheno-K\"ahler} if $\partial \overline{\partial} \omega^{n-2} = 0$ (see \cite{JY1993});
    \item {\em balanced} if $d \omega^{n-1} = 0$ (see \cite{M1982}).
\end{itemize}
Such Hermitian metrics have been extensively studied (see \cite{AG1986}, \cite{AI2001}, \cite{B1989}, \cite{FPS2004}, \cite{FT2011}, \cite{JY1993}, \cite{MT2001}, \cite{M1982} and the references therein). As already reminded, a large class of examples of compact complex manifolds that admit such metrics but no K\"ahler metrics is provided by {\em solvmanifolds} (resp. {\em nilmanifolds}): by a solvmanifold (resp. nilmanifold) we mean a compact quotient of a simply connected, connected solvable (nilpotent) Lie group by a lattice. 

In \cite{S1976}, D. Sullivan introduced the notion of {\em transversality} in the context
of cone structures, namely a continuous field of 
cones of $p$-vectors on a manifold. Later, L. Alessandrini and M. Andreatta (see \cite{AA1987, AA1987Erratum}) defined the {\em $p$-K\"ahler manifolds} as complex manifolds endowed with a closed $(p,p)$-transverse form. Such forms on a complex manifold are called {\em $p$-K\"ahler structures}. Note that closed $(1,1)$-transverse forms, respectively, closed $(n-1,n-1)$-transverse forms correspond to K\"ahler metrics, respectively, to balanced metrics. For $1<p<n-1$, closed $(p,p)$-transverse forms have non metric meaning, that is if the $p$-power of the fundamental form $\omega$ of a Hermitian metric $g$ is closed, then $d\omega=0$, namely $g$ is a K\"ahler metric. The notions of $p$-{\em pluriclosed manifolds} (resp. structures), respectively, $p$-{\em symplectic manifolds} (resp. structures) can be introduced in a similar way (see \cite{A2017} and Definitions \ref{Definitions of p-structures}, \ref{Definitions of p-manifolds}).

In the context of non-K\"ahler geometry, many authors have investigated the possible coexistence of different metric structures on the same complex manifold. Some of them focused on the possible coexistence of SKT or Astheno-K\"ahler metrics and balanced metrics. 
 
In this article, we focus on  {\em $p$-symplectic} (resp. {\em $p$-K\"ahler}, {\em $p$-pluriclosed}) structures. For $p = 1,n-1$, $p$-pluriclosed structures coincide with SKT and {\em Gauduchon} metrics, respectively. Furthermore, $1$-symplectic structures are referred to as {\em Hermitian-symplectic structures} in \cite{EFV2012, FV2019}, \cite{ST2010}, while $(n-1)$-symplectic structures are called {\em strongly Gauduchon metrics} in \cite{COUV2016}, \cite{LU2017}, \cite{P2013} and \cite{X2015}. 

In \cite{AI2001} and \cite{MT2001} the authors proved that the same metric cannot be both SKT and balanced or Astheno-K\"ahler and balanced unless it is a K\"ahler metric. Thus, some authors have studied whether a compact complex manifold could admit two different metrics, $g$ and $g'$, such that $g$ is SKT or Astheno-K\"ahler and $g'$ is balanced. In \cite{FV2016}, it was proven that on a nilmanifold equipped with an invariant complex structure the existence of a SKT metric $g$ and a balanced metric $g'$ forces the nilmanifold to be a torus. In particular, the authors conjectured that a similar behaviour holds for every compact complex manifold. Meanwhile, in \cite{FGV2019}, \cite{LU2017} and \cite{ST2023} are constructed examples of compact complex manifolds that admit two different metrics such that one of them is balanced and the other one is Astheno-K\"ahler. Moreover, we mention that, in \cite{EFV2012, FV2019}, it is showed that $1$-symplectic structures cannot exist on nilmanifolds equipped with an invariant complex structure unless they are tori.

The aim of this paper is to investigate the existence of $(n-1)$-symplectic structures on nilmanifolds equipped with invariant complex structure. Specifically, we show that there are families of nilmanifolds that admit both $(n-1)$-symplectic structures together with special Hermitian metric structures (see Theorem \ref{3-symplectic-astheno-skt} and Theorem \ref{4-symplectic and Astheno and 2-pluriclosed}). We mention that in \cite{COUV2016}, the authors classified all possible $6$-dimensional nilmanifolds equipped with an invariant complex structure that can admit $2$-symplectic structures. In particular, they describe which $6$-dimensional nilmanifolds have $2$-symplectic structures but no balanced metrics. Moreover, in \cite[Proposition 2.1]{OUV2017}, the authors studied all possible $6$-dimensional non-K\"ahler nilmanifolds that can admit SKT metrics and $2$-symplectic structures. They show that, in these cases, all invariant Hermitian metrics are SKT and their second exterior power is the $(2,2)$-component of a $2$-symplectic structure.

Furthermore, we investigate the existence of $p$-symplectic (resp. $p$-K\"ahler, $p$-pluriclosed) structures on {\em compact holomorphically parallelizable manifolds} of solvable type.

The paper is organized as follows: in Section \ref{Section 2} and \ref{Section 3}, we establish the notation used throughout the article and provide basic definitions concerning the positivity of $(p,p)$-forms and the concept of $p$-symplectic (resp. $p$-K\"ahler, $p$-pluriclosed) manifolds.

In Section \ref{Section 4}, we extend the obstructions for the existence of $p$-K\"ahler and $p$-pluriclosed structures provided in \cite{HMT2023} and \cite{ST2023} to the non-invariant case. Furthermore, we provide an obstruction for the existence of $p$-symplectic structures on compact complex manifolds (Lemma \ref{No-existence of p-symplectic}). 

Section \ref{Section 5} is devoted to the existence of $p$-symplectic (resp. $p$-K\"ahler, $p$-pluriclosed) structures on compact holomorphically parallelizable manifolds. In \cite{AB1991}, the authors showed that, on such manifolds, the existence of a $p$-symplectic structure is equivalent to the existence of a $p$-K\"ahler one. We first show that on compact holomorphically parallelizable manifolds the existence of such structures is equivalent to the existence of a $p$-pluriclosed one (Theorem \ref{Equivalence between p-forms on holomorphically parallelizable manifolds}). Subsequently, we focus our attention on compact holomorphically parallelizable manifolds of type $IV$ and $V$, according to the classification by I. Nakamura (see \cite{N1975}). We show that, a compact holomorphically parallelizable manifold of type $IV$ or $V$, which is not a torus or of type $V \, 3)$, cannot admit a $p$-symplectic (resp. $p$-K\"ahler, $p$-pluriclosed) structure (Propositions \ref{No p-structures on type IV}, \ref{No p-structures on Type V}). 

In Section \ref{Section 6} we study three families of nilmanifolds. Firstly, we consider the {\em Fino-Parton-Salamon nilmanifolds}, i.e., the $6$-dimensional nilmanifolds introduced in \cite{FPS2004} and we directly prove that they admit both $2$-symplectic structures and SKT metrics (Theorem \ref{2-symplectic-skt}). As already reminded, in \cite[Proposition 2.1] {OUV2017} the authors gave a complete classification of all 6-nilmanifolds that admit both an SKT metric and a $2$-symplectic structure.  
\newline
The second one highlights the coexistence of $3$-symplectic structures together with SKT and Astheno-K\"ahler metrics on the family of $8$-dimensional nilmanifolds introduced in \cite{FT2011} (Corollary \ref{3-symplectic-astheno-skt}). In particular, according to N. Enrietti, A. Fino and L. Vezzoni (see \cite{EFV2012, FV2019}), if a compact complex nilmanifold is $1$-symplectic, then it is K\"ahler; the same conclusion does not hold for $p=n-1$, indeed any complex nilmanifold belonging to one of the above two families has a $(n-1)$-symplectic structure with no balanced metrics, for $n=3$ and $n=4$, respectively.\newline
The third one provides $10$-dimensional nilmanifolds \cite{ST2023} carrying $4$-symplectic structures, Astheno-K\"ahler metrics and $2$-pluriclosed structures (Corollary \ref{4-symplectic and Astheno and 2-pluriclosed}).

In \cite{T2015}, V. Tosatti introduced the notion of {\em non K\"ahler Calabi-Yau manifolds}, as compact complex manifolds with vanishing first Bott-Chern class $c_{1}^{BC}(M)$ in $H^{1,1}_{BC}(M;\R)$ (in \cite{FJS2023}, the authors considered compact complex manifolds with first Chern class vanishing in $H^{2}_{dR}(M;\R)$). In particular, for such manifolds $c_1(M)=0$ in $H^2_{dR}(M;\R)$. He gave some explicit examples and constructions of compact complex manifolds, some with $c_{1}^{BC}(M)=0$ and others with $c_{1}^{BC}(M)\neq 0$, proving that the class of non K\"ahler Calabi-Yau manifolds is strictly contained in the class of compact complex manifolds with trivial first Chern class. 
\newline 
We show that the non-K\"ahler, generalized K\"ahler,  solvmanifold constructed in \cite{FT2009} has vanishing first Chern class but non vanishing first Bott-Chern class. This provides another example beside to those ones obtained in \cite{T2015} of compact complex manifolds with $c_{1}^{BC}(M)\neq 0$ in $H^{1,1}_{BC}(M;\R)$ and  $c_1(M)=0$ in $H^2_{dR}(M;\R)$.
\vskip.3truecm
{\em \underline{Acknowledgments:}} The authors would like to thank Valentino Tosatti, Gueo Grantcharov and Asia Mainenti for useful comments and remarks. Many thanks are also due to Luis Ugarte for bringing to their attention the references \cite{COUV2016} and \cite{OUV2017}.  

\section{Preliminaries}\label{Section 2}
Let $(M,J)$ be a complex manifold of complex dimension $n$, i.e., a real $2n$-dimensional manifold equipped with $J \in \text{End}(TM)$ such that $J^2=-\hbox{\rm Id}_{TM}$ and the Nijenhuis tensor associated to $J$ vanishes. We call $J$ the complex structure of $M$.

The endomorphism $J$ can be extended by $\C$-linearity to the complexified tangent bundle $T_{\C}M \doteq TM \otimes \C$. Thus, $T_{\C}M$ decomposes as a direct sum of the eigenbundles associated to $\pm i$, namely, 
\begin{equation*}
    T_{\C}M = T^{1,0}M \oplus T^{0,1}M, 
\end{equation*}
where $T^{1,0}M \doteq \{X \in T_{\C}M \, | \, JX =i X\}$ and $T^{0,1}M \doteq \{X \in T_{\C}M \, | \, JX =-i X\}$. The exterior bundle of complex $r$-forms, $\Lambda^{r}_{\C}M \doteq \Lambda^{r} T_{\C}^{\ast}M$, decomposes as
\begin{equation*}
    \Lambda^{r}_{\C}M = \oplus_{p+q=r} \Lambda^{p,q}M,
\end{equation*} 
where $\Lambda^{p,q}M \doteq \Lambda^{p}(T^{1,0}M)^{\ast} \otimes \Lambda^{q}(T^{0,1}M)^{\ast}$. Furthermore, we adopt the following notation 
\begin{equation*}
    \Lambda^{r}_{\R} M \doteq \{ \eta \in \Lambda^{r}_{\C}M \, | \, \overline{\eta} = \eta\}, \quad \Lambda^{p,p}_{\R} M \doteq \{\eta \in \Lambda^{p,p}_{\C}M \, | \, \overline{\eta} = \eta \}.
\end{equation*}

Let us denote by $T_{x}M$ and $T_{x}^{\ast}M$ the tangent space of $M$ at $x \in M$ and its dual, respectively. Let $\{ \varphi_{i} \}_{i=1}^{n}$ be a basis of $\Lambda^{1,0} (T_{x}^{\ast}M \otimes \C)$, and let us denote $\varphi^{I_p} \doteq \varphi^{i_1} \wedge \dots \wedge \varphi^{i_p}$, where $I_{p} = (i_{1},\dots, i_{p})$. Thus, 
\begin{equation*}
    \{\varphi^{I_{p}} \wedge \overline{\varphi}^{J_{q}} \, | \, i_{1}<\dots < i_{p}, j_{1}< \dots < j_{q}  \}
\end{equation*}
is a basis of $\Lambda^{p,q} (T_{x}^{\ast}M \otimes \C)$. 

A basis for $\Lambda^{p,p}_{\R}(T^{\ast}_{x}M) \doteq \{ \eta \in \Lambda^{p,p}(T^{\ast}_{x}M \otimes \C) \, | \, \overline{\eta} = \eta\}$ is given by $\{ \sigma_{p} \varphi^{I_{p}} \wedge \overline{\varphi}^{I_{p}} \, | \, i_{1}<\dots < i_{p}\}$, where $\sigma_{p} \doteq i^{p^{2}}2^{-p}$, moreover 
\begin{equation*}
    \text{Vol} \doteq (\frac{i}{2} \varphi^{1} \wedge \overline{\varphi}^{1}) \wedge \dots \wedge (\frac{i}{2} \varphi^{n} \wedge \overline{\varphi}^{n}) = \sigma_{n} \varphi^{1}\wedge \dots \wedge \varphi^{n} \wedge \overline{\varphi}^{1} \wedge \dots \wedge \overline{\varphi}^{n},
\end{equation*}
is a volume form.

\begin{definition}
     $\psi \in \Lambda^{p,0}(T^{\ast}_{x}M \otimes \C)$ is said to be {\em simple} if $ \psi = \psi^{1} \wedge \dots \wedge \psi^{p}$, where $\psi^{i} \in \Lambda^{1,0}(T^{\ast}_{x}M \otimes \C)$, for $i=1, \dots, p$.
\end{definition}

Let $\psi \in \Lambda^{n,n}_{\R}(T^{\ast}_{x}M)$, then it is called positive (resp. strictly positive) if $\psi = a \text{Vol}$, where $a \geq 0$ ($a > 0$).
We adopt the following definitions of positivity for a $(p,p)$-form:
\begin{definition}\label{Definitions of positivity of forms}
    Let $\psi \in  \Lambda^{p,p}_{\R}(T^{\ast}_{x}M)$, then 
        \begin{enumerate}
            \item $\psi$ is called {\em transverse}, if 
                \begin{equation*}
                    \sigma_{n-p} \psi \wedge \beta \wedge \overline{\beta} 
                \end{equation*}
                is strictly positive, $\forall \beta \in \Lambda^{n-p,0}(T^{\ast}_{x}M \otimes \C)$ simple, $\beta \neq 0$; 
            \item $\psi$ is called {\em positive definite}, if 
                \begin{equation*}
                    \sigma_{n-p} \psi \wedge \beta \wedge \overline{\beta}
                \end{equation*}
                is strictly positive, $\forall\beta \in \Lambda^{n-p,0}(T^{\ast}_{x}M \otimes \C)$, $\beta \neq 0$;
            \item $\psi$ is called {\em strictly strongly positive}, if there exists $\{\psi^{1},\dots, \psi^{r}\}$, $\psi^{i} \in \Lambda^{p,0}(T^{\ast}_{x}M \otimes \C)$ simple, such that 
                \begin{equation*}
                    \psi = \sigma_{p} \sum_{i=1}^{r} \psi^{i} \wedge \overline{\psi}^{i},
                \end{equation*}
            and 
            \begin{equation*}
                    \sigma_{n-p} \psi \wedge \beta \wedge \overline{\beta}
            \end{equation*}
            is strictly positive, $\forall\beta \in \Lambda^{n-p,0}(T^{\ast}_{x}M \otimes \C)$, $\beta \neq 0$.
        \end{enumerate}
\end{definition}

\begin{remark}
    It can be easily checked that the three definitions of positivity coincide if $p=1,n-1$ (see \cite{HK1974}). 
\end{remark}
In this paper by a {\em solvmanifold} we mean a compact quotient $M \doteq \Gamma\backslash G$ of a connected, simply conntected, solvable Lie group $G$ by a lattice, i.e., a uniform discrete subgroup $\Gamma \subseteq G$. In particular, when $G$ is nilpotent we call $M$ a {\em nilmanifold}. We recall that a $n$-dimensional complex manifold $M$ is said to be {\em holomorphically parallelizable} if $T^{1,0}M$ is trivial as a holomorphic bundle, that is there exist $n$ holomorphic $1$-forms on $M$ which are linearly independent at every point of $M$. A complex structure $J$ on a solvmanifold $M=\Gamma\backslash G$ is said to be {\em invariant} if it stems from a left-invariant complex structure on the Lie group $G$. We recall that a solvmanifold with invariant complex structure is called {\em almost Abelian} if its Lie algebra admits an abelian ideal of codimension $1$. 

Furthermore, a nilmanifold equipped with an invariant complex structure $J$ is called $J$-{\em nilpotent} if the ascending series $\{\mathfrak{g}_{i}^{J}\}$ defined by
\begin{equation*}
    \mathfrak{g}^{J}_{0} \doteq 0, \quad \mathfrak{g}^{J}_{i} \doteq \{X \in \mathfrak{g} \enskip | \enskip [X, \mathfrak{g}]\subseteq \mathfrak{g}^{J}_{i-1}, \quad [J X, \mathfrak{g}] \subseteq \mathfrak{g}^{J}_{i-1} \}
\end{equation*}
satisfies $\mathfrak{g}^{J}_{k}= \mathfrak{g}$ for some $k > 0$, or equivalently (see \cite{CFGU1997}, \cite{CFGU2000}), if there exists $\{\varphi^{i}\}_{i=1, \dots n}$ coframe of invariant $(1,0)$-forms such that for $i=1,\dots n$,
\begin{equation*}
    d \varphi^{i} = \Lambda^{2} \langle \varphi^{1}, \dots, \varphi^{i-1}, \varphi^{\overline{1}}, \dots, \varphi^{\overline{i-1}} \rangle.
\end{equation*}

\section{$p$-symplectic, $p$-k\"ahler and $p$-pluriclosed manifolds}\label{Section 3}
We are ready to state the following definitions of $p$-{\em K\"ahler}, $p$-{\em pluriclosed}, $p$-{\em symplectic} structures on a complex manifold. The requirement is that the transverse condition must hold pointwise. Naturally, these definitions can be generalized for almost complex manifolds (see \cite{HMT2023}, \cite{ST2023}).

\begin{definition}\label{Definitions of p-structures}
    Let $(M,J)$ be a complex manifold of complex dimension $n$, and let $1 \leq p \leq n$. Let $\psi$ be a real $(p,p)$-form, then 
    \begin{enumerate}
        \item $\psi$ is called a $p$-{\em K\"ahler} structure if it is $d$-closed and $\psi_{x} \in \Lambda^{p,p}_{\R}(T_{x}^{\ast}M)$ is transverse $\forall x \in M$;
        \item $\psi$ is called a $p$-{\em pluriclosed} structure if it is $\partial \overline{\partial}$-closed and $\psi_{x} \in \Lambda^{p,p}_{\R}(T_{x}^{\ast}M)$ is transverse $\forall x \in M$;
        \item $\Psi \in \Lambda^{2p}_{\R}M$ is called a $p$-{\em symplectic} structure if $\Psi$ is $d$-closed and its $(p,p)$-component, $\Psi^{p,p}_{x} \in \Lambda^{p,p}_{\R}(T_{x}^{\ast}M)$, is transverse $\forall x \in M$.
    \end{enumerate}
\end{definition}

\begin{definition}\label{Definitions of p-manifolds}
    The triple $(M,J,\psi)$, where $M$ is a complex manifold, $J$ is a complex structure and $\psi \in \Lambda^{2p}_{\R}M$ is a $p$-symplectic (resp. $p$-K\"ahler, $p$-pluriclosed) structure is called a $p$-{\em symplectic} (resp. {\em $p$-K\"ahler}, {\em $p$-pluriclosed}) {\em manifold}.  
\end{definition}

\begin{remark}
    From the definitions above, we can easily see that $p$-K\"ahler structures are contained in $p$-symplectic ones.
\end{remark}

\begin{remark}\label{p-structures for p = 1,n-1}
    Let $p=1$, then $1$-K\"ahler structures are equivalent to K\"ahler metrics, while $1$-pluriclosed structures are equivalent to  Strong-K\"ahler with torsion metrics. Moreover, $1$-symplectic structures have been called Hermitian-symplectic structures in \cite{EFV2012, FV2019}, \cite{ST2010}.

    Let $p=n-1$, from \cite{M1982}, we know that $(n-1,n-1)$-forms which are transverse are the $(n-1)$-power of Hermitian metrics. Hence, $(n-1)$-K\"ahler structures are equivalent to balanced metrics. Furthermore, P. Gauduchon in \cite{G1977} proved that if $(M,J)$ is a compact complex manifold of real dimension $2n > 2$, then given an Hermitian metric $g$ there exists a unique conformally equivalent metric $\widetilde{g}$ whose fundamental form $\widetilde{\omega}$ is such that $\widetilde{\omega}^{n-1}$ is a $(n-1)$-pluriclosed structure. A Hermitian metric on $M$ such that the $(n-1)$-power of its fundamental form is a $(n-1)$-pluriclosed structure is called a Gauduchon metric. Hence, from \cite{M1982}, $(n-1)$-pluriclosed structures are equivalent to Gauduchon metrics. 

    In \cite{P2013} and \cite{X2015}, $(n-1)$-symplectic manifolds are referred to as strongly Gauduchon manifolds. Indeed, the existence of a $(n-1)$-symplectic form is equivalent to the existence of a Hermitian metric $\omega$ such that $\partial \omega^{n-1}$ is $\overline{\partial}$-exact. Metrics satisfying the latter condition are called strongly Gauduchon metrics.
\end{remark}

In \cite[Proposition 4.1]{X2015}, the author proved that strongly Gauduchon manifolds are stable under small holomorphic deformations. Similarly, we can easily observe that $p$-symplectic manifolds are also stable under small holomorphic deformations. Indeed, let $\pi: \chi \to M$ be a holomorphic deformation of compact complex manifolds with base manifold $M$, then we can adapt the proof of \cite{X2015} to prove that $p$-symplectic manifolds are stable under small deformation. 

\section{Obstructions}\label{Section 4}
Let $(M,J)$ be a compact complex manifold of complex dimension $n$. In \cite{HMT2023} and \cite{ST2023}, the authors provided obstructions to the existence of $p$-K\"ahler and $p$-pluriclosed structures. Here, we extend these obstructions to the non-invariant case. Subsequently, we state a lemma that provides an obstruction to the existence of $p$-symplectic structures. From now on, a form that satisfies the first definition of Definition \ref{Definitions of p-structures} will be referred to as either a $p$-K\"ahler structure or a $p$-K\"ahler form. The same terminology will be used for $p$-pluriclosed structures and $p$-symplectic structures.

First of all, we notice that the obstruction to the existence of $p$-K\"ahler forms, provided in \cite[Proposition 3.4]{HMT2023}, holds true also if 
\begin{equation*}
    (d \beta)^{n-p,n-p} = \sum_{i} f_{i} \psi^{i} \wedge \overline{\psi}^{i},
\end{equation*}
where $\beta$ is a $(2n-2p-1)$-form, $\psi^{i}$ are simple $(n-p,0)$-forms and $f_{i}$ are functions on $M$ which have the same sign.

Furthermore, we can generalize the obstruction to the existence of $p$-pluriclosed forms, provided in \cite[Lemma 3.5]{ST2023}, in a similar manner. For the sake of completeness we provide a different proof for complex manifold in the following lemma.

\begin{lemma}\label{No-existence of p-pluriclosed}
    Let $(M,J)$ be a compact complex manifold of complex dimension $n$. Suppose that $\beta$ is a $(2n-2p-2)$-form such that
    \begin{equation*}
        (\partial \overline{\partial} \beta)^{n-p,n-p} = \sum_{i} f_{i} \psi^{i} \wedge \overline{\psi}^{i},
    \end{equation*}
     where $\psi^{i}$ are simple $(n-p,0)$-forms, and $f_{i}$ are functions on $M$ which have the same sign. Then, $(M,J)$ does not admit a $p$-pluriclosed form.
\end{lemma}
\begin{proof}
    Let us suppose that $\Omega$ is a $p$-pluriclosed form and that $\{f_{i}\}$ are positive. Then 
    \begin{equation}\label{First equation of Lemma No-existence of p-pluriclosed}
        0 = \int_{M} \sigma_{n-p} d\overline{\partial}(\Omega \wedge \beta) = \int_{M} \sigma_{n-p} \partial \overline{\partial}(\Omega \wedge \beta).
    \end{equation}
    Since $\Omega$ is a $(p,p)$-form, the only terms of $\beta$ that give a contribute in $\Omega \wedge \beta$ are the components with degrees $(n-p,n-p-2)$, $(n-p-1,n-p-1)$, and $(n-p-2,n-p)$.

    Let us consider the $(n-p,n-p-2)$ component of $\beta$, denoted by $\beta^{n-p,n-p-2}$. Since $\Omega \wedge \beta^{n-p,n-p-2}$ is a $(n,n-2)$-form, then $\partial \overline{\partial}(\Omega \wedge \beta^{n-p,n-p-2}) = 0$. In the same way $\partial \overline{\partial} ( \Omega \wedge \beta^{n-p-2,n-p})=0$. 

    Thus, the only term of $\beta$ which give a non-zero contribute is $\beta^{n-p-1,n-p-1}$. So, if $\widetilde{\beta} \doteq \beta^{n-p-1,n-p-1}$, then from \eqref{First equation of Lemma No-existence of p-pluriclosed} 
    \begin{equation*}
        \begin{split}
            0 & = \int_{M} \sigma_{n-p} \partial \overline{\partial}(\Omega \wedge \widetilde{\beta}) = \int_{M} \sigma_{n-p} \partial \big( \overline{\partial}\Omega \wedge \widetilde{\beta}) \big)  + \int_{M} \sigma_{n-p} \partial \big( \Omega \wedge \overline{\partial} \widetilde{\beta} \big) \\
            & =  \int_{M} \sigma_{n-p} \partial \big( \overline{\partial}\Omega \wedge \widetilde{\beta}) \big) + \int_{M} \sigma_{n-p} \partial \Omega \wedge \overline{\partial} \widetilde{\beta} + \int_{M} \sigma_{n-p} \Omega \wedge \partial \overline{\partial} \widetilde{\beta}.
        \end{split}
    \end{equation*}
    Moreover,
    \begin{equation*}
        \begin{split}
            \int_{M} \sigma_{n-p} \partial \big( \overline{\partial} \Omega \wedge \widetilde{\beta} \big) & = \int_{M} \sigma_{n-p} \partial \big( \overline{\partial} \Omega \wedge \widetilde{\beta} \big) + \int_{M} \sigma_{n-p} \overline{\partial} \big( \overline{\partial} \Omega \wedge \widetilde{\beta} \big) \\
            & = \int_{M} \sigma_{n-p} d \big( \overline{\partial} \Omega \wedge \widetilde{\beta} \big) = 0,
        \end{split}
    \end{equation*}
    where we could add $\int_{M} \sigma_{n-p} \overline{\partial} \big( \overline{\partial}\Omega \wedge \widetilde{\beta}) \big)$ because $\overline{\partial}\Omega \wedge \widetilde{\beta}$ is a $(n-1,n)$-form. \newline 
    Furthermore
    \begin{equation*}
        \begin{split}
            \int_{M} \sigma_{n-p} \partial \Omega \wedge \overline{\partial} \widetilde{\beta} & = \int_{M} \sigma_{n-p} \overline{\partial} \partial \Omega \wedge \widetilde{\beta} - \int_{M} \sigma_{n-p} \overline{\partial} \big( \partial \Omega \wedge \widetilde{\beta} \big)  \\
            & = - \int_{M} \sigma_{n-p} \overline{\partial} \big( \partial \Omega \wedge \widetilde{\beta} \big) - \int_{M} \sigma_{n-p} \partial \big( \partial \Omega \wedge \widetilde{\beta} \big) \\
            & = - \int_{M} \sigma_{n-p} d \big( \partial \Omega \wedge \widetilde{\beta} \big) = 0.
        \end{split}
    \end{equation*}
    Since $(\partial \overline{\partial} \beta)^{n-p,n-p}=\partial \overline{\partial} \widetilde{\beta}$, we get 
    \begin{equation*}
        0 = \int_{M} \sigma_{n-p} \Omega \wedge \partial \overline{\partial} \widetilde{\beta} > 0,
    \end{equation*}
    which is absurd. 
\end{proof}
In the following lemma we provide an obstruction to the existence of $p$-symplectic forms.
\begin{lemma}\label{No-existence of p-symplectic}
    Let $(M,J)$ be a compact complex manifold of complex dimension $n$. Suppose that $\beta$ is a $(2n-2p-1)$-form such that 
    \begin{equation*}
        d \beta = \sum_{i} f_{i} \psi^{i} \wedge \overline{\psi}^{i},
    \end{equation*}
    where $\psi^{i}$ are simple $(n-p,0)$-forms, and $f_{i}$ are functions on $M$ which have the same sign. Then $(M,J)$ does not admit a $p$-symplectic form. 
\end{lemma}
\begin{proof}
    Let us suppose that $\{f_{i}\}$ are positive and that $\Omega$ is a $p$-symplectic form. Then 
    \begin{equation*}
        0 = \int_{M} \sigma_{n-p} d(\Omega \wedge \beta) = \int_{M} \sigma_{n-p} \Omega \wedge d\beta = \sum_{i} \int_{M} f_{i} \Omega \wedge \sigma_{n-p} \psi^{i} \wedge \overline{\psi}^{i} >0,
    \end{equation*}
    where the last relation holds true because $\Omega^{p,p}$, which is the only component of $\Omega$ that give a contribute to the integral, is transverse. Thus, all the integrals are strictly positive.
\end{proof}

\section{$p$-symplectic structures on holomorphically parallelizable manifolds}\label{Section 5}
In this Section, we study the existence of $p$-K\"ahler, $p$-pluriclosed and $p$-symplectic structures on compact holomorphically parallelizable manifolds. We recall that H. C. Wang, in \cite{W1954}, proved that a compact complex manifold $M$ is holomorphically parallelizable if and only if $M = \Gamma \backslash G$, where $G$ is a connected, simply-connected, complex Lie group and $\Gamma$ is lattice of $G$.

Firstly, we show that, on such manifolds, the existence of a real $(p,p)$-form that satisfies one of the definitions in \ref{Definitions of p-structures} implies the existence of real $(p,p)$-forms satisfying the other definitions of \ref{Definitions of p-structures}. Subsequently, we focus on compact holomorphically parallelizable manifold of type $IV$ or $V$ (see \cite{N1975} or the classification below). Let $M$ be a compact holomorphically parallelizable manifold of type $IV$ or $V$ which is not a torus or of type $V3)$. We already know, from \cite{AG1986}, that every left-invariant Hermitian metric on a compact holomorphically parallelizable manifold is a balanced metric. Hence, every left-invariant Hermitian metric provides $(n-1)$-symplectic (resp. $(n-1)$-K\"ahler, $(n-1)$-pluriclosed) structure. 
We will prove that $M$ cannot admit a $p$-symplectic (resp. $p$-K\"ahler, $p$-pluriclosed) structure if $1 < p < \text{dim}_{\C} M-1$.

Finally, we show that, if a compact holomorphically parallelizable manifold satisfies the $\partial \overline{\partial}$-Lemma, then it is a K\"ahler manifold (Theorem \ref{Compact holomorphically parallelizable manifolds and del-del bar lemma}). The theorem follows from an easy application of the characterization provided in Theorem \ref{Equivalence between p-forms on holomorphically parallelizable manifolds}. 

\begin{theorem}\label{Equivalence between p-forms on holomorphically parallelizable manifolds}
    Let $M$ be a compact holomorphically parallelizable manifold of complex dimension $n$. Let $1 \leq p \leq n$, then the following statements are equivalent:
    \begin{enumerate}
        \item $M$ admits a $p$-K\"ahler structure;
        \item $M$ admits a $p$-symplectic structure;
        \item $M$ admits a $p$-pluriclosed structure;
        \item there are no $(n-p,0)$-forms $\xi$, $\xi \neq 0$ such that $\xi$ is exact, simple and holomorphic.
    \end{enumerate}
\end{theorem}
\begin{proof}
    The equivalence of {\em (1), (2)} and {\em (4)} is proved in \cite{AB1991}. We just prove that {\em (1)} implies {\em (3)} and {\em (3)} implies {\em (4)}.

    Let us suppose that $\Omega$ is a $p$-K\"ahler structure. Hence, $\Omega$ is a $(p,p)$-form which is transverse and $d \Omega = 0$. Then 
    \begin{equation*}
        d \Omega = \partial \Omega + \overline{\partial} \Omega = 0,
    \end{equation*}
    so $\overline{\partial} \Omega = 0$ and $\partial \overline{\partial} \Omega = 0$. Hence, $\Omega$ is a $p$-pluriclosed structure, thus {\em (1)} implies {\em (3)}.

    Suppose that $\Omega$ is a $p$-pluriclosed structure and that there exists a $(n-p,0)$-form $\xi$ such that $\xi \neq 0$, $\xi$ is holomorphic, simple and $\xi = d \alpha$, where $\alpha$ is a $(n-p-1)$-form. Then 
    \begin{equation}\label{Symmetrization of xi}
        \xi = \partial \alpha^{n-p-1,0},
    \end{equation}
    where $\alpha^{n-p-1,0}$ is the $(n-p-1,0)$ component of $\alpha$. \newline
    Given $\{\varphi^{1}, \dots, \varphi^{n}\}$ the holomorphic $(1,0)$-coframe of $M$, then $\{\varphi^{i_{1}} \wedge \dots \wedge \varphi^{i_{n-p-1}} \, | \newline i_{1} < \dots < i_{n-p-1}\}$ is a basis for $(n-p-1,0)$-forms. By H. C. Wang (see \cite{W1954}), $M$ is the compact quotient of a complex Lie group $G$ by a lattice $\Gamma \subseteq G$, then by a symmetrization process applied to both the left and right hand sides of \eqref{Symmetrization of xi}, we get 
    \begin{equation}
        \xi = \partial \widetilde{\alpha},
    \end{equation}
    where $\widetilde{\alpha}$ is the symmetrization of $\alpha^{n-p-1,0}$, $\xi$ being left-invariant. Moreover $\widetilde{\alpha}$ is holomorphic, hence, $\overline{\partial}\widetilde{\alpha}=0$. 
    Thus, the transversality of $\Omega$ implies that 
    \begin{equation*}
        \begin{split}
            0 < \int_{M} \sigma_{n-p} \, \Omega \wedge \partial \widetilde{\alpha} \wedge \overline{\partial \widetilde{\alpha}} & = \int_{M} \sigma_{n-p} \, \partial(\Omega \wedge  \widetilde{\alpha} \wedge \overline{\partial \widetilde{\alpha}}) - \int_{M} \sigma_{n-p} \, \partial \Omega \wedge  \widetilde{\alpha} \wedge \overline{\partial \widetilde{\alpha}} \\
            & = \int_{M} \sigma_{n-p} \, d(\Omega \wedge \widetilde{\alpha} \wedge \overline{\partial \widetilde{\alpha}}) - \int_{M} \sigma_{n-p} \, \partial \Omega \wedge \widetilde{\alpha} \wedge \overline{\partial \widetilde{\alpha}},
        \end{split}
    \end{equation*}
    where on the first line we use that $\partial \overline{\partial \widetilde{\alpha}} = 0$, while on the second line we can add $\int_{M} \sigma_{n-p} \, \overline{\partial}(\Omega \wedge \widetilde{\alpha} \wedge \overline{\partial \widetilde{\alpha}})$ because $\Omega \wedge \widetilde{\alpha} \wedge \overline{\partial \widetilde{\alpha}}$ is a $(n-1,n)$-form.
    
    Thus 
    \begin{equation*}
        \begin{split}
            0 < \int_{M} \sigma_{n-p} \, \Omega \wedge \partial \widetilde{\alpha} \wedge \overline{\partial \widetilde{\alpha}} &  = - \int_{M} \sigma_{n-p} \, \partial \Omega \wedge \widetilde{\alpha} \wedge \overline{\partial \widetilde{\alpha}} \\
            & = \pm  \int_{M} \sigma_{n-p} \, \overline{\partial} (\partial \Omega \wedge \widetilde{\alpha} \wedge \overline{\widetilde{\alpha}}) \pm \int_{M} \sigma_{n-p} \, \overline{\partial} \partial \Omega \wedge  \widetilde{\alpha} \wedge \overline{\widetilde{\alpha}}  \\
            & = \pm \int_{M} \sigma_{n-p} \, d(\partial \Omega \wedge \widetilde{\alpha} \wedge \widetilde{\alpha}) \pm \int_{M} \sigma_{n-p} \, \overline{\partial}\partial \Omega \wedge \widetilde{\alpha} \wedge \widetilde{\alpha} = 0,
        \end{split}
    \end{equation*}
    The signs of the integrals of the above equation depends on the dimension of the manifold and on $p$. We do not analyze every single case because the signs of the integrals are not relevant to the argument. Finally, {\em (3)} implies {\em (4)}.
\end{proof}
\begin{remark}
    The fact that a $p$-K\"ahler structure is also a $p$-pluriclosed structure holds true for every complex manifold. 
\end{remark}

The next part of this Section is devoted to studying the existence of $p$-K\"ahler, $p$-symplectic and $p$-pluriclosed structures on compact holomorphically parallelizable manifolds, in accordance with the classification by I. Nakamura (see \cite{N1975}). For the sake of completeness, we remind below the classification of complex solvable Lie algebras of dimension $4$ and $5$  see \cite[pp.108-109]{N1975}.
\renewcommand{\arraystretch}{1.5}
\begin{center}
Type $IV$
\begin{tabular}{  m{1em} | m{1.5cm}| m{9.5cm} | } 
  \hline
  1) & abelian & $d \varphi^{1} = d \varphi^{2} = d \varphi^{3} = d \varphi^{4} = 0$; \\ 
  \hline
  2) & nilpotent & $d \varphi^{1} = d \varphi^{2} = d \varphi^{3} = 0, \enskip d \varphi^{4} = - \varphi^{23}$; \\
  \hline
  3) & nilpotent & $d \varphi^{1} = d \varphi^{2} = 0, \enskip  d\varphi^{3} = - \varphi^{12}, \enskip d \varphi^{4} = -2 \varphi^{13}$; \\
  \hline
  4) & solvable & $d \varphi^{1} = d \varphi^{2} = 0, \enskip d \varphi^{3} = \varphi^{23}, \enskip d \varphi^{4} = - \varphi^{24}$; \\
  \hline
  5) & solvable & $d \varphi^{1}=0, \enskip d \varphi^{2} = \varphi^{12} , \enskip d \varphi^{3} = \alpha \varphi^{13} , \enskip d \varphi^{4} = - (1 + \alpha)\varphi^{14}$; \\
  \hline
  6) & solvable & $d \varphi^{1}=0, \enskip d \varphi^{2} = \varphi^{12}, \enskip d \varphi^{3} = - \varphi^{13}, \enskip d \varphi^{4} = - \varphi^{23}$; \\
  \hline 
  7) & solvable & $d \varphi^{1}=0, \enskip d \varphi^{2} = \varphi^{12}, \enskip \varphi^{3} = - 2 \varphi^{13}, \enskip d \varphi^{4}= \varphi^{14}- \varphi^{12}$; \\
  \hline
\end{tabular}
\end{center}
where $\alpha \in \C$ such that $\alpha(1+\alpha) \neq 0$.
\renewcommand{\arraystretch}{1.5}
\begin{center}
Type $V$
\begin{tabular}{  m{1em} | m{1.5cm}| m{9.5cm} | } 
  \hline
  1) & abelian & $d \varphi^{1} = d \varphi^{2} = d \varphi^{3} = d \varphi^{4} = d \varphi^{5} = 0$; \\ 
  \hline
  2) & nilpotent & $d \varphi^{1} = d \varphi^{2} =d \varphi^{3} = d \varphi^{4} = 0, \enskip  d\varphi^{5} = - \varphi^{34}$; \\
  \hline
  3) & nilpotent & $d \varphi^{1} = d \varphi^{2} = d \varphi^{3} = d \varphi^{4} = 0, \enskip d \varphi^{5} = - \varphi^{13} - \varphi^{24}$; \\
  \hline
  4) & nilpotent & $d \varphi^{1} = d \varphi^{2} = d \varphi^{3} = 0, \enskip d \varphi^{4} = -\varphi^{12}, \enskip d \varphi^{5} = - \varphi^{13}$; \\
  \hline
  5) & nilpotent & $d \varphi^{1} = d \varphi^{2} = d \varphi^{3} = 0, \enskip d \varphi^{4} = -\varphi^{23} , \enskip d \varphi^{5} = -2 \varphi^{24}$; \\
  \hline
  6) & nilpotent & $d \varphi^{1} = d \varphi^{2} = d \varphi^{3} = 0, \enskip d \varphi^{4} = -\varphi^{12}, \enskip d \varphi^{5} = -2 \varphi^{14} - \varphi^{23}$; \\
  \hline 
  7) & solvable & $d \varphi^{1} = d \varphi^{2} = d \varphi^{3} = 0, \enskip d \varphi^{4} = \varphi^{34}, \enskip \varphi^{5} = - \varphi^{35}$; \\
  \hline
  8) & nilpotent & $d \varphi^{1} = d \varphi^{2} = 0, \enskip d \varphi^{3} = -\varphi^{12}  , \enskip d \varphi^{4} = -2 \varphi^{13} , \enskip d \varphi^{5} = -2 \varphi^{23}$; \\
  \hline 
  9) & nilpotent & $d \varphi^{1} = d \varphi^{2} = 0, \enskip d \varphi^{3} = - \varphi^{12} , \enskip d \varphi^{4} = -2 \varphi^{13}, \enskip d \varphi^{5} = -3 \varphi^{14}$; \\
  \hline
  10) & nilpotent & $d \varphi^{1} = d \varphi^{2} = 0, \enskip d \varphi^{3} = - \varphi^{12} , \enskip d \varphi^{4} = -2 \varphi^{13},$ \\
  & & $ d \varphi^{5} = -3 \varphi^{14} - \varphi^{23}$; \\
  \hline 
  11) & solvable  & $d \varphi^{1} = d \varphi^{2} = 0, \enskip d \varphi^{3} = -\varphi^{12} , \enskip d \varphi^{4} = \varphi^{14}, \enskip d \varphi^{5} = \varphi^{15};$ \\
  \hline 
  12) & solvable & $d \varphi^{1} = d \varphi^{2} = 0, \enskip d \varphi^{3} = \varphi^{13} , \enskip d \varphi^{4} = \varphi^{24},$ \\
  & & $d \varphi^{5} = - (\varphi^{1} + \varphi^{2}) \wedge \varphi^{5};$ \\
  \hline
\end{tabular}
\end{center}

\renewcommand{\arraystretch}{1.5}
\begin{center}
Type $V$
\begin{tabular}{  m{1em} | m{1.5cm}| m{9.5cm} | } 
    \hline
    13) & solvable & $d \varphi^{1} = d \varphi^{2} = 0, \enskip d \varphi^{3} = \varphi^{23} , \enskip d \varphi^{4} = \alpha \varphi^{24}, \enskip d \varphi^{5} = -(1 + \alpha)\varphi^{25}$; \\
    \hline
    14) & solvable & $d \varphi^{1} = d \varphi^{2} = 0, \enskip d \varphi^{3} = \varphi^{13} , \enskip d \varphi^{4} = -2 \varphi^{14}, \enskip d \varphi^{5} = \varphi^{15} - \varphi^{13}$; \\
    \hline 
    15) & solvable & $d \varphi^{1} = d \varphi^{2} = 0, \enskip d \varphi^{3} = \varphi^{23} , \enskip d \varphi^{4} = - \varphi^{24}, \enskip d \varphi^{5} = - \varphi^{34}$; \\
    \hline
    16) & solvable & $d \varphi^{1} = d \varphi^{2} = 0, \enskip d \varphi^{3} = \varphi^{13} , \enskip d \varphi^{4} = - \varphi^{14}, \enskip d \varphi^{5} = -\varphi^{34}-\varphi^{12}$; \\
    \hline
    17) & solvable & $d \varphi^{1} = 0, \enskip d \varphi^{2} = \varphi^{12}, \enskip d \varphi^{3} = \gamma \varphi^{13}, \enskip d \varphi^{4} = \beta \varphi^{14},$ \\
    & & $d \varphi^{5} = -(1 + \gamma + \beta) \varphi^{15}$; \\
    \hline
    18) & solvable & $d \varphi^{1} = 0, \enskip d \varphi^{2} = -3 \varphi^{12}, \enskip d \varphi^{3} = \varphi^{13}, \enskip d \varphi^{4} = \varphi^{14} - \varphi^{13},$ \\
      & & $d \varphi^{5} = \varphi^{15} - \varphi^{13}$; \\
      \hline
    19) & solvable & $d \varphi^{1} = 0, \enskip d \varphi^{2} = \varphi^{12}, \enskip d \varphi^{3} = - \varphi^{13}, \enskip d \varphi^{4} = \varphi^{14} - \varphi^{12},$ \\
    & & $d \varphi^{5} = - \varphi^{15} - \varphi^{13}$; \\ 
    \hline 
    20) & solvable & $d \varphi^{1} = 0,    \enskip d \varphi^{2} = \varphi^{12}, \enskip d \varphi^{3} = \varphi^{13} - \varphi^{12}, \enskip d \varphi^{4} = \eta \varphi^{14},$ \\
          & & $d \varphi^{5} = -(2 + \eta)\varphi^{15}$; \\
    \hline
\end{tabular}
\end{center}
where $\alpha, \beta, \gamma, \eta \in \C$ such that $\alpha(1+\alpha) \neq 0$, $\gamma \beta(1 + \gamma + \beta) \neq 0$ and $\eta(2 + \eta) \neq 0$. As the author points out, there are no compact holomorphically parallelizable manifolds of types $IV\, 7), V \, 15)$ and $V\, 18)$. Meanwhile, he does not know if compact parallelizable manifolds of types $IV \, 5), V \, 11), V \, 13), V \, 16), V \, 19), V \, 20)$ exist.

\medskip 

From \cite{AG1986}, we know that all compact holomorphically parallelizable manifolds carry balanced metrics. Meanwhile, the only ones of type $IV$ and $V$ that carry a K\"ahler metric are the tori of complex dimension $4$ and $5$. We aim to investigate what happens in the other cases.

An interesting holomorphically parallelizable manifold is $\eta \beta_{2n+1}$, which is a nilmanifold that generalizes the Iwasawa manifold to higher dimensions (see \cite{AB1991}). In according to the classification by I. Nakamura, $\eta \beta_{5} \doteq \Gamma \backslash G$, where $\Gamma$ is a lattice, is a holomorphically parallelizable nilmanifold of type $V \, 3)$. The Lie group $G$ can be identified with $(\C^{5}, \ast)$, where 
\begin{equation*}
    (w_{1}, \dots, w_{5}) \ast (z_{1},\dots, z_{5}) \doteq (w_{1} + z_{1}, \dots, w_{5} + z_{5} + w_{1}z_{3} + w_{2}z_{4}),
\end{equation*} 
and $\Gamma$ is the lattice formed by Gaussian integers. Moreover, we can easily see that we can choose a basis of left-invariant $(1,0)$-forms $\{\varphi^{1}, \dots, \varphi^{5}\}$ such that 
\begin{equation*}
    d \varphi^{1} = \dots = d \varphi^{4} = 0, \quad d \varphi^{5} = - \varphi^{13} - \varphi^{24}.
\end{equation*}
L. Alessandrini and G. Bassanelli, in \cite{AB1991}, proved that $\eta \beta_{2n+1}$ carries a $p$-K\"ahler form for $n+1 \leq p \leq 2n+1$ and no $p$-K\"ahler forms for $1 \leq p \leq n$.

In order to provide an alternative proof of the existence of a $3$-K\"ahler form on $\eta \beta_{5}$, we provide a characterization of transverse $(2,2)$-forms in $\C^{4}$. This characterization, together with Proposition \ref{Transversal (2,2)-form}, is a part of a forthcoming paper by F. Fagioli and A. Mainenti (\cite{FM2024}), and it was the topic of a discussion of the first author with A. Mainenti.

Let us consider $\C^{4}$ equipped with the following basis of $(1,0)$-forms $\{\omega^{1}, \dots, \omega^{4}\}$. The volume form is $\text{Vol} = \sigma_{4} \omega^{1234 \overline{1234}} = \frac{1}{2^{4}} \omega^{1234 \overline{1234}}$. We set a basis for $\Lambda^{2,0}\C^{4}$ and we denote it by $\{\Omega_{1}, \dots, \Omega_{6}\}$, where
\begin{equation*}
    \begin{split}
        & \Omega^{1} \doteq \varphi^{12}, \quad \Omega^{2} \doteq \varphi^{13}, \quad \Omega^{3} \doteq \varphi^{14}, \\
        &\Omega^{4} \doteq \varphi^{23}, \quad \Omega^{5} \doteq - \varphi^{24}, \quad \Omega^{6} \doteq \varphi^{34}.
    \end{split}
\end{equation*}
Then 
\begin{equation*}
    \Omega^{j} \wedge \Omega^{k} = \begin{cases}
        \omega^{1234}, \quad \text{if} \, \,  k = 7-j, \\
        0, \quad \text{otherwise},
    \end{cases}
\end{equation*}
thus, we can associate to a $(2,2)$-form $\alpha$ a matrix $6 \times 6$, $A = (a_{jk})_{j,k=1,\dots, 6}$, defined by 
\begin{equation*}
    \alpha = \sum_{j,k} a_{jk} \Omega_{j} \wedge \overline{\Omega_{k}}.
\end{equation*}
Once we have fixed these notations, we can observe that \cite[Theorem 3]{BP2013} can be easily adapted to our definition of transversality. The authors proved that, for $(2,2)$-forms in $\C^{4}$, the weakly positive transversality, i.e., $\alpha \wedge \sigma_{2} \beta \wedge \overline{\beta} = c \text{Vol}$ with $c \geq 0$, is equivalent to 
\begin{equation*}
    \overline{z} A z^{t} \geq 0, \quad \forall z \in C^{6} \, \, \text{such} \, \, \text{that} \, \, z \in Q,
\end{equation*}
where $Q: z_{1}z_{6} + z_{2}z_{5} + z_{3}z_{4} = 0$. The proof can be easily adapted to demonstrate that $\alpha$ is transverse if and only if  
\begin{equation*}
    \overline{z} A z^{T} > 0, \quad \forall z \in Q, z \neq 0.
\end{equation*}
The following proposition adapts \cite[Proposition 4]{BP2013} to transversality. Furthermore, it introduces two additional families of transverse $(2,2)$-forms that were not highlighted in \cite{BP2013}. Since we will use this proposition, for the sake of completeness, we provide a proof of it.

\begin{proposition}\label{Transversal (2,2)-form}
    The $(2,2)$-form 
    \begin{equation}
        \Omega_{a} = \sum_{l=1}^{6}\Omega_{l} \wedge \overline{\Omega_{l}} + a \Omega_{i} \wedge \overline{\Omega_{j}} + \overline{a} \Omega_{j} \wedge \overline{\Omega_{i}}
    \end{equation}
    for $a \in \C$ and $(i,j)=(1,6),(2,5),(3,4)$ is transverse if and only if $|a| < 2$.
\end{proposition}
\begin{proof}
    Here, we provide a proof for $(i,j)=(2,5)$, the other cases follow in a similar manner. 

    Suppose that $|a| < 2$. We aim to show that $\overline{z}Az^{t} > 0$, for $z \in Q$. 
    
    Since 
    \begin{equation*}
        \begin{split}
            \overline{z}Az^{t} & = |z|^{2} + 2 \text{Re}(a\overline{z}_{2} z_{5}) \\
            & \geq 2|z_{1}z_{6}| + 2|z_{2}z_{5}| + 2|z_{3}z_{4}| + 2 \text{Re}(a\overline{z}_{2} z_{5}) \\
            & \geq 2|z_{1}z_{6} + z_{3}z_{4}| + 2|z_{2}z_{5}| - 2|a z_{2} z_{5}| \\
            & = 4 |z_{2}z_{5}| - 2|a z_{2} z_{5}| = 2(2-|a|)|z_{2}z_{5}|,
        \end{split}
    \end{equation*}
    then, in order to conclude that $\overline{z}Az^{t} > 0$ in $Q$, we just have to study the case $|z_{2}z_{5}| = 0$.
    
    If $|z_{2}z_{5}| = 0$, then $|z_{2}| = 0$ or $|z_{5}| = 0$ and in both cases $ \overline{z}Az^{t} = |z|^{2}$, thus $\overline{z}Az^{t} > 0$, unless $z=0$.
    Hence, if $|a| < 2$ then $\Omega_{a}$ is tranverse.

    Furthermore, if $\Omega_{a}$ is transverse, then by taking $z \in Q$ such that $\overline{z_{2}}z_{5} = - \overline{a}$, $|z_{2}|=|z_{5}|, |z_{1}| = |z_{6}|$, $|z_{3}| = |z_{4}|$, $z_{1}z_{6}$ and $z_{3}z_{4}$ linearly dependent, we get 
    \begin{equation*}
        \begin{split}
            \overline{z}Az^{t} & = |z|^{2} + 2 \text{Re}(a\overline{z}_{2} z_{5}) = |z|^{2} - 2 |a|^{2} \\
            & = 2|z_{1}z_{6}| + 2|z_{2}z_{5}| + 2|z_{3}z_{4}| - 2 |a|^{2} \\
            & = 2|z_{1}z_{6} + z_{3}z_{4}| + 2|z_{2}z_{5}| - 2|a|^{2} \\
            & =  4|a| - 2|a|^{2} = 2 |a| (2 - |a|),
        \end{split}
    \end{equation*}
    thus the thesis. 
    The other cases follow in the same way.
\end{proof}

Here, we give a different proof of $3$-K\"ahlerianity of the complex manifold $\eta \beta_{5}$ by explicitly constructing a $3$-K\"ahler form on it.

\begin{proposition}
    $\eta \beta_{5}$ admits a $3$-K\"ahler form.
\end{proposition}
\begin{proof} The only non trivial structure equation is given by
    $$d \varphi^{5} = - \varphi^{13} - \varphi^{24};$$ 
    then 
    \begin{equation*}
        d \varphi^{5 \overline{5}} = - \varphi^{13 \overline{5}} - \varphi^{24 \overline{5}} + \varphi^{5 \overline{13}} + \varphi^{5 \overline{24}}. 
    \end{equation*}
    Thus 
    \begin{equation*}
        \begin{split}
            & d \varphi^{ijk \overline{ijk}} = 0, \quad \forall i,j,k = 1,\ldots, 4,\quad i < j < k, \\
            & d \varphi^{rs5 \overline{rs5}} = 0,
        \end{split}
    \end{equation*}
    for $(r,s) = (1,2),(1,4),(2,3),(3,4)$, and 
    \begin{equation*}
        d \varphi^{135 \overline{135}} = \varphi^{1234 \overline{135}} - \varphi^{135 \overline{1234}}, \quad d \varphi^{245 \overline{245}} = \varphi^{1234 \overline{245}} - \varphi^{245 \overline{1234}} .
    \end{equation*}
    Define 
    \begin{equation*}
        \Omega \doteq \sigma_{3} \big( \sum_{i<j<k} \varphi^{ijk \overline {ijk}} - \varphi^{135 \overline{245}} - \varphi^{245 \overline{135}} \big).
    \end{equation*}
    \underline{Claim}: 
    $\Omega$ is a $3$-K\"ahler form.\newline
    Since 
    \begin{equation*}
        d \varphi^{245 \overline{135}} = \varphi^{1234 \overline{135}} - \varphi^{245 \overline{1234}}, \quad d \varphi^{135 \overline{245}} = \varphi^{1234 \overline{245}} - \varphi^{135 \overline{1234}},
    \end{equation*}
    then 
    \begin{equation*}
        d \Omega = \sigma_{3} d \big(\sum_{i<j<k}\varphi^{ijk \overline{ijk}} \big) - \sigma_{3} d \big(\varphi^{135 \overline{245}} + \varphi^{245 \overline{135}} \big) = 0,
    \end{equation*}
    hence $\Omega$ is $d$-closed. 
    
    We have to prove that $\Omega$ is transverse. In order to prove it we use Proposition \ref{Transversal (2,2)-form}. Let $\{V_{1}, \dots, V_{5}\}$ be the dual basis of $\{\varphi^{1}, \dots, \varphi^{5}\}$, then, we can decompose the complexified Lie algebra $\mathfrak{g}_{\C}$ of $\eta \beta_{5}$ in the following way 
    \begin{equation*}
        \mathfrak{g}_{\C} = \mathfrak{h}_{\C} \oplus \langle V_{5}, \overline{V_{5}} \rangle,
    \end{equation*}
    where $\mathfrak{h}_{\C} \doteq \mathfrak{h}_{1,0} \oplus \overline{\mathfrak{h}_{1,0}}$, where $\mathfrak{h}_{1,0} \doteq \langle V_{1}, \dots, V_{4}\rangle$. 
    
    Since $\Omega \doteq \sigma_{3} \big( \sum_{i<j<k} \varphi^{ijk \overline {ijk}} - \varphi^{135 \overline{245}} - \varphi^{245 \overline{135}} \big)$, then 
    \begin{equation*}
        \Omega = \Omega|_{\mathfrak{h}_{\C}} + \sigma_{3} \varphi^{5 \overline{5}} \wedge \big(F + \Theta \big),
    \end{equation*}
    where 
    \begin{equation*}
        \begin{split}
            & \Omega|_{\mathfrak{h}_{\C}} \doteq \sigma_{3} (\varphi^{123 \overline{123}} + \varphi^{124 \overline{124}} + \varphi^{134 \overline{134}} + \varphi^{234 \overline{234}} ), \\
            & F \doteq \big(\varphi^{12 \overline{12}} + \varphi^{13 \overline{13}} + \varphi^{14 \overline{14}} + \varphi^{23 \overline{23}} + \varphi^{24 \overline{24}} + \varphi^{34 \overline{34}} \big), \quad \Theta \doteq - \big(\varphi^{13 \overline{24}} + \varphi^{24 \overline{13}} \big).
        \end{split}
    \end{equation*}
    Up to a scalar positive real term, $\Omega|_{\mathfrak{h}_{\C}}$ is the third power of the fundamental form of the diagonal metric on $\mathfrak{h}_{\C}$, thus it is transverse on $\mathfrak{h}_{\C}$, while $F+\Theta$ is the $(2,2)$-form described in Proposition \ref{Transversal (2,2)-form} with $a = 1$, thus it is transverse on $\mathfrak{h}_{\C}$. 
    
    Let $\eta \in \Lambda^{2,0}(\mathfrak{g}^{\ast}_{\C})$ be simple, then we have the following cases: 
    \begin{enumerate}
        \item $\eta \in \Lambda^{2,0}(\mathfrak{h}_{\C}^{\ast})$;
        \item $\eta \in \Lambda^{2,0}(\mathfrak{g}_{\C}^{\ast}) \setminus \Lambda^{2,0}(\mathfrak{h}_{\C}^{\ast})$.
    \end{enumerate}
    Thus, in the first case 
    \begin{equation*}
        \eta \doteq (a_{1} \varphi^{1} + \dots + a_{4} \varphi^{4}) \wedge (b_{1} \varphi^{1} + \dots + b_{4} \varphi^{4}),
    \end{equation*}
    so $\Omega|_{\mathfrak{h}_{\C}} \wedge \sigma_{2} \eta \wedge \overline{\eta} = 0$, because it is a $(5,5)$-form on $\mathfrak{h}_{\C}$. While, $\sigma_{3} \varphi^{5 \overline{5}} \wedge \big(F + \Theta \big) \wedge \sigma_{2} \eta \wedge \overline{\eta} = c \text{Vol}$, with $c>0$ because 
    \begin{equation*}
        \big(F + \Theta \big) \wedge \sigma_{2} \eta \wedge \overline{\eta} = \widetilde{c}\text{Vol}_{\mathfrak{h}_{\C}},
    \end{equation*}
    where $\text{Vol}_{\mathfrak{h}_{\C}}$ is the volume form on $\mathfrak{h}_{\C}$ and $\widetilde{c} > 0$ because $(F+ \Theta)$ is transverse on $\mathfrak{h}_{\C}$. 
    
    In the second case we have that, without lost of generality, we can suppose that $\eta = (a_{1} \varphi^{1} + \dots + a_{4} \varphi^{4}) \wedge a_{5} \varphi^{5}$. Thus 
    \begin{equation*}
        \Omega|_{\mathfrak{h}_{\C}} \wedge \sigma_{2} \eta \wedge \overline{\eta} = \sigma_{3} \sigma_{2} |a_{5}|^{2}(|a_{1}|^{2} + |a_{2}|^{2} + |a_{3}|^{2} + |a_{4}|^{2}) \varphi^{12345 \overline{12345}},
    \end{equation*}
    hence $\Omega|_{\mathfrak{h}_{\C}} \wedge \sigma_{2} \eta \wedge \overline{\eta} = c \text{Vol}$, where $c > 0$. Furthermore, 
    \begin{equation*}
        \sigma_{3} \varphi^{5 \overline{5}} \wedge \big(F + \Theta \big) \wedge \sigma_{2} \eta \wedge \overline{\eta} = 0.
    \end{equation*}
    Thus, the thesis follows and $\Omega$ is an invariant $3$-K\"ahler form. 
\end{proof}
Now we are able to state the following propositions.
\begin{proposition}\label{No p-structures on type IV}
    The only compact holomorphically parallelizable manifold of type $IV$ which admits a real $(2,2)$-form satisfying one of the definitions in \ref{Definitions of p-structures} is the torus.
\end{proposition}
\begin{proof}
    The complex structure of a holomorphically parallelizable nilmanifolds is nilpotent. Thus, from  \cite[Theorem 3.7]{FM2023} we can conclude that holomorphically parallelizable nilmanifolds of types $2)$ or $3)$ do not admit any $2$-K\"ahler form.

    Furthermore, $4)$ and  $5)$ are almost abelian holomorphically parallelizable solvmanifolds. Hence, from \cite[Theorem 4.2]{FM2023}, we can conclude that they do not admit any $2$-K\"ahler form. 
    
    In order to prove that a holomorphically parallelizable solvmanifold of type $6)$ does not admit any 2-K\"ahler form, we use \cite[Proposition 3.4]{HMT2023}. Indeed
    \begin{equation*}
        d \varphi^{12 \overline{2}} = \varphi^{12 \overline{12}},
    \end{equation*}
    hence, there are no $2$-K\"ahler forms.
    
    Thus, we have proved that they do not admit any $2$-K\"ahler form. The thesis of the proposition follows from Theorem \ref{Equivalence between p-forms on holomorphically parallelizable manifolds}.
\end{proof}
\begin{proposition}\label{No p-structures on Type V}
    There are no compact holomorphically parallelizable manifolds of type $V$, which are not of type $V \, 1)$ or $V \, 3)$, that admit a real $(p,p)$-form, for $1 < p < 4$, satisfying one of the definitions in \eqref{Definitions of p-structures}.
\end{proposition}
\begin{proof}
    Arguing as in the precedent proposition, from \cite[Theorem 3.8]{FM2023} we can conclude that there are no holomorphically parallelizable nilmanifolds of types $2)$, $3)$, $4)$, $5)$, $6)$, $8)$, $9)$ and $10)$ that admit a $2$-K\"ahler form.

    Furthermore, \cite[Theorem 2.3]{ST2022} allows us to conclude that there are no holomorphically parallelizable nilmanifolds of types $8)$, $9)$ and $10)$ that admit a $3$-K\"ahler form. In order to prove that there are no holomorphically parallelizable nilmanifolds of type $2)$, $3)$, $4)$, $5)$ and $6)$ that admit a $3$-K\"ahler form we use \cite[Proposition 3.4]{HMT2023}. For instance let us consider the case $5)$: 
    \begin{equation*}
        d \varphi^{23 \overline{4}} = - \varphi^{23 \overline{23}},
    \end{equation*}
    hence, it does not admit any $3$-K\"ahler form. Cases $2)$, $3)$, $4)$ and $6)$ follow in a similar way. 

    The last part of the proof concerns the study of compact holomorphically parallelizable solvmanifolds. 

    Compact holomorphically parallelizable solvmanifolds of types $7)$, $11)$, $13)$, $14)$, $17)$, $19)$ and  $20)$ are almost abelian. Hence, from \cite[Theorem 4.2]{FM2023}, we can conclude that they do not admit any $3$-K\"ahler forms. We can prove that there are no solvmanifolds of the aforementioned types that admit a $2$-K\"ahler form by using \cite[Proposition 3.4]{HMT2023}. For instance let us consider $14)$, then 
    \begin{equation*}
        d \varphi^{123 \overline{23}} = - \varphi^{123 \overline{123}}.
    \end{equation*}
    The other cases follow in a similar way.
    In order to prove that compact holomorphically parallelizable solvmanifolds of types $12)$ and $16)$ do not admit any $2$ or $3$-K\"ahler form we use \cite[Proposition 3.4]{HMT2023} as it was showed before. 

    Finally, the thesis follows from Theorem \ref{Equivalence between p-forms on holomorphically parallelizable manifolds}.
\end{proof}
We recall that a manifold is said to satisfy the $\partial \overline{\partial}$-Lemma if 
\begin{equation*}
    \text{ker} \, \partial \cap \text{ker} \, \overline{\partial} \cap \text{im} \, d = \text{im} \, \partial \overline{\partial},
\end{equation*}
where the operators are considered on the correct space of $(p,q)$-forms. \newline
The following result is well known, for the sake of completeness we give a proof by using Theorem \ref{Equivalence between p-forms on holomorphically parallelizable manifolds}.
\begin{theorem}\label{Compact holomorphically parallelizable manifolds and del-del bar lemma}
    Let $M$ be a compact holomorphically parallelizable manifold. If $M$ satisfies the $\partial \overline{\partial}$-Lemma, then it is a K\"ahler manifold.  
\end{theorem}
\begin{proof}
    Let us suppose that the complex dimension of $M$ is $n$.
    
    Firstly, we show that if $M$ satisfies the $\partial \overline{\partial}$-Lemma, then, $\forall p = 1, \dots, n$, there are no $(p,0)$-forms $\xi$, $\xi \neq 0$ such that $\xi$ is $d$-exact, simple and holomorphic. Indeed if such $(p,0)$-form exists, then 
    \begin{equation*}
        \xi = d \alpha = \partial \alpha^{p-1,0}, \quad \overline{\partial} \xi = 0, \quad \partial \xi = 0,
    \end{equation*}
    where $\alpha$ is a $(p-1)$-form and $\alpha^{p-1,0}$ is the $(p-1,0)$-component of $\alpha$. Hence, from the $\partial \overline{\partial}$-Lemma, it is $\partial \overline{\partial}$-exact but this is absurd. 

    Thus, from Theorem \ref{Equivalence between p-forms on holomorphically parallelizable manifolds}, we can conclude that $M$ admits a $p$-K\"ahler form for every $p$, hence it is K\"ahler. 
\end{proof}

\section{$(n-1)$-symplectic forms on nilmanifolds}\label{Section 6}
In this Section we provide families of $(n-1)$-symplectic nilmanifolds for $n=3,4,5$ endowed with an invariant complex structure. As previously reminded in the Introduction, we show that an invariant $(n-1)$-symplectic structure can coexist with an Astheno-K\"ahler metric. We recall that, in \cite[Proposition 5.1]{COUV2016}, Ceballos, Otal, Ugarte and Villacampa showed that a nilmanifold equipped with an invariant complex structure admits a $(n-1)$-symplectic structure if and only if it admits an invariant $(n-1)$-symplectic structure. Hence, we can focus on invariant $(n-1)$-symplectic structures.

Furthermore, we provide an example of a compact complex, non-K\"ahler manifold which has a non-vanishing first Bott-Chern class, but a vanishing first Chern class.
\subsection{$2$-symplectic form on $6$-dimensional nilmanifolds}
In \cite{FPS2004}, A. Fino, M. Parton and S. Salamon proved that, for a $6$-dimensional nilmanifold $(M \doteq \Gamma \backslash G,J)$ with an invariant complex structure, the property of being SKT is satisfied by all Hermitian metrics or by none. Furthermore, they proved that it is satisfied if and only if $J$ admits a basis of $(1,0)$-forms $\{\alpha^{1},\alpha^{2},\alpha^{3}\}$ such that
\begin{equation}\label{Structure equations for 6-dimensional Nilmanifolds with SKT}
    \begin{cases}
        d \alpha^{1} = 0, \\
        d \alpha^{2} = 0, \\
        d \alpha^{3} = A \alpha^{\overline{1}2} + B \alpha^{\overline{2}2} + C \alpha^{1 \overline{1}} + D \alpha^{1 \overline{2}} + E \alpha^{12},
    \end{cases}
\end{equation}
where $A,B,C,D,E \in \C$ satisfy
\begin{equation*}
    |A|^{2} + |D|^{2} + |E|^{2} + 2 \text{Re}(\overline{B}C) = 0.
\end{equation*}
Here, we directly construct $2$-symplectic structures on such manifolds. 

\begin{remark}
    Let $(M,J)$ be a $6$-dimensional complex nilmanifold equipped with an invariant complex structure which admits a basis of $(1,0)$-forms $\{\alpha_{1}, \alpha_{2}, \alpha_{3}\}$ such that 
    \begin{equation}\label{Structure equations for 6-dimensional Nilmanifolds}
        \begin{cases}
            d \alpha^{1} = 0, \\
            d \alpha^{2} = 0, \\
            d \alpha^{3} = A \alpha^{\overline{1}2} + B \alpha^{\overline{2}2} + C \alpha^{1 \overline{1}} + D \alpha^{1 \overline{2}} + E \alpha^{12},
        \end{cases}
    \end{equation}
    where $A,B,C,D,E \in \C$. 
    An invariant, transverse, real $4$-form on $(M,J)$ is given by 
    \begin{equation}\label{Invariant, transverse, real, 4-form}
        \Psi  = \lambda^{3,1} + \eta + \overline{\lambda^{3,1}},
    \end{equation}
    where $\eta$ is a real $(2,2)$-form which must be transverse and 
    \begin{equation}\label{(3,1)-forms of the 2-symplectic form on FPS}
        \lambda^{3,1} \doteq L \alpha^{123 \overline{1}} + M \alpha^{123 \overline{2}} + N \alpha^{123 \overline{3}}, \quad L,M,N \in \C.
    \end{equation}
    From \cite{M1982}, we know that $\eta$ must be the power of a Hermitian metric.
    In order to completely describe the possible invariant $2$-symplectic forms on $(M,J)$, let us consider $\omega$ a generic left-invariant Hermitian metric, i.e., 
\begin{equation}\label{Gauduchon metric}
    \omega \doteq \frac{i}{2} (r^{2} \alpha^{1 \overline{1}}+ s^{2} \alpha^{2 \overline{2}} + t^{2}\alpha^{3 \overline{3}}) + \frac{u}{2}\alpha^{1 \overline{2}} - \frac{\overline{u}}{2}\alpha^{2 \overline{1}} + \frac{v}{2}\alpha^{1 \overline{3}} - \frac{\overline{v}}{2}\alpha^{3 \overline{1}} + \frac{w}{2}\alpha^{2 \overline{3}} - \frac{\overline{w}}{2}\alpha^{3 \overline{2}}, 
\end{equation}
where $r,s,t \in \R$ and $u,v,w \in \C$ satisfy the following conditions:
\begin{equation*}
    r^{2} > 0, \quad r^{2} s^{2} - |u|^{2} > 0, \quad r^{2}s^{2}t^{2} - 2 \text{Re}(i u \overline{v} w) > r^{2}|w|^{2} + s^{2}|v|^{2} + t^{2} |u|^{2}.
\end{equation*}
Furthermore, from a straightforward computation
\begin{equation*}
    \begin{split}
        \omega^{2} = & -\frac{1}{2} (r^{2} s^{2} \alpha^{1 \overline{1}2 \overline{2}} + r^{2} t^{2} \alpha^{1 \overline{1}3 \overline{3}} + s^{2} t^{2} \alpha^{2 \overline{2}3 \overline{3}}) + \frac{|u|^{2}}{2} \alpha^{1 \overline{1} 2 \overline{2}} + \frac{|v|^{2}}{2} \alpha^{1 \overline{1} 3 \overline{3}}  \\
        & + \frac{|w|^{2}}{2} \alpha^{2 \overline{2} 3 \overline{3}} + \frac{u \overline{v}}{2} \alpha^{1 \overline{1}3 \overline{2}} + \frac{\overline{u} v}{2} \alpha^{1 \overline{1}2 \overline{3}} - \frac{u w}{2} \alpha^{2 \overline{2}1 \overline{3}} - \frac{\overline{uw}}{2} \alpha^{2 \overline{2}3 \overline{1}} \\
        & + \frac{v \overline{w}}{2} \alpha^{3 \overline{3}1 \overline{2}} + \frac{\overline{v} w}{2} \alpha^{3 \overline{3}2 \overline{1}} + \frac{i}{2} (r^{2}w \alpha^{1 \overline{1}2 \overline{3}} - r^{2} \overline{w} \alpha^{1 \overline{1} 3 \overline{2}} \\
        & + s^{2}v \alpha^{2 \overline{2} 1 \overline{3}} -s^{2}\overline{v}\alpha^{2 \overline{2} 3 \overline{1}})  + \frac{i}{2}(t^{2} u \alpha^{3 \overline{3} 1 \overline{2}} - t^{2} \overline{u} \alpha^{3 \overline{3} 2 \overline{1}}). 
    \end{split}
\end{equation*}
\end{remark}
 
\begin{theorem}\label{existence of 2-symplectic}
    Let $(M = \Gamma \backslash G,J)$ be a $6$-dimensional nilmanifold equipped with an invariant complex structure $J$. Suppose that $J$ admits a basis $\{\alpha^{1}, \alpha^{2}, \alpha^{3}\}$ of $(1,0)$-forms satisfying \eqref{Structure equations for 6-dimensional Nilmanifolds}. Let $\Psi$ be an invariant, transverse, real, $4$-form, i.e., $\Psi \doteq \lambda^{3,1} + \omega^{2} + \overline{\lambda^{1,3}}$, where $\lambda^{3,1}$ and $\omega$ are, respectively, defined as in \eqref{(3,1)-forms of the 2-symplectic form on FPS} and \eqref{Gauduchon metric}. 
    
    Then, $\Psi$ is a $2$-symplectic form if and only if 
    \begin{equation}\label{Condition for existence of 2-symplectic forms on 6-dimensional Nilmanifolds}
        -N \overline{E} + \frac{1}{2} \big( -r^{2} t^{2} \overline{B} + s^{2} t^{2} \overline{C} + |v|^{2} \overline{B} - |w|^{2}\overline{C} + i t^{2} u \overline{D} + it^{2}\overline{uA} + v \overline{wD} - \overline{v}w \overline{A} \big)= 0.
    \end{equation}
\end{theorem}

\begin{proof}
    From \eqref{Structure equations for 6-dimensional Nilmanifolds}, we get 
    \begin{equation*}
        d \alpha^{i \overline{i} j \overline{3}} = d \alpha^{i \overline{i} 3 \overline{j}} = d \alpha^{1 \overline{1} 2 \overline{2}}=0,
    \end{equation*}
    for $i=1,2$, $j=1,2$, and 
    \begin{equation*}
        \begin{split}
            d\alpha^{3 \overline{3}} & = -A \alpha^{2 \overline{13}} - B \alpha^{2 \overline{23}} + C \alpha^{1 \overline{13}} + D \alpha^{1\overline{23}} + E \alpha^{12 \overline{3}} \\
            & \quad + \overline{A} \alpha^{13 \overline{2}} + \overline{B} \alpha^{23 \overline{3}} - \overline{C} \alpha^{13 \overline{1}} - \overline{D} \alpha^{23 \overline{1}} - \overline{E} \alpha^{3 \overline{12}},
        \end{split}
    \end{equation*}
    then
    \begin{equation*}
        \begin{split}
            & d \alpha^{1 \overline{1} 3 \overline{3}} =  B \alpha^{12 \overline{1} \overline{23}} + \overline{B} \alpha^{123  \overline{1} \overline{3}}, \quad  d \alpha^{2 \overline{2}3 \overline{3}} = - C \alpha^{12 \overline{123}} - \overline{C} \alpha^{123 \overline{12}}, \\
            & d \alpha^{3 \overline{3} 1 \overline{2}} = -A \alpha^{12 \overline{123}} + \overline{D} \alpha^{123 \overline{12}}, \quad d\alpha^{3 \overline{3}2 \overline{1}} = - \overline{A}\alpha^{123 \overline{12}} + D \alpha^{12 \overline{123}}.
        \end{split}
    \end{equation*}
    Thus  
    \begin{equation}\label{Lambda (3,1)}
        \begin{split}
            d \omega^{2} & = \frac{1}{2} \big(- r^{2}t^{2} \overline{B} + s^{2}t^{2} \overline{C} + |v|^{2}\overline{B} - |w|^{2} \overline{C} + it^{2}u \overline{D} + i t^{2}\overline{uA} + v \overline{wD} - \overline{v}w \overline{A} \big) \alpha^{123 \overline{12}} \\
            & + \frac{1}{2} \big( - r^{2}t^{2} B + s^{2}t^{2} C + |v|^{2}B - |w|^{2} C - it^{2}\overline{u} D - i t^{2}u A + \overline{v} w D - v\overline{w} A \big) \alpha^{12 \overline{123}}.
        \end{split}
    \end{equation}
    Moreover,
    \begin{equation*}
        d \lambda^{3,1} = -N \alpha^{123} \wedge (\overline{A} \alpha^{1 \overline{2}} + \overline{B} \alpha^{2 \overline{2}} + \overline{C} \alpha^{\overline{1}1} + \overline{E} \alpha^{\overline{12}} + \overline{D} \alpha^{\overline{1}2}) = -N \overline{E} \alpha^{123 \overline{12}},
    \end{equation*}
    and 
    \begin{equation*}
        d \overline{\lambda^{3,1}} = - \overline{N} E \alpha^{12 \overline{123}}.
    \end{equation*}
    Thus, $d \Psi = 0$ if and only if 
    \begin{equation*}
        -N \overline{E} + \frac{1}{2} \big( -r^{2} t^{2} \overline{B} + s^{2} t^{2} \overline{C} + |v|^{2} \overline{B} - |w|^{2}\overline{C} + i t^{2} u \overline{D} + i t^{2} \overline{uA} + v \overline{wD} - \overline{v}w \overline{A} \big)= 0
    \end{equation*}
\end{proof}

As a consequence we get the following 
\begin{theorem}\label{2-symplectic-skt}
    Let $(M = \Gamma \backslash G,J)$ be a $6$-dimensional nilmanifold equipped with an invariant complex structure. Let $\omega$ be the fundamental form of the diagonal Hermitian metric and $\Psi$ be an invariant, real, transverse $4$-form defined by 
    \begin{equation*}
        \Psi \doteq \lambda^{3,1} + \omega^{2} + \overline{\lambda^{3,1}},
    \end{equation*}
    where $\lambda^{3,1}$ is defined as in \eqref{(3,1)-forms of the 2-symplectic form on FPS}.
    The diagonal metric is SKT and simultaneously $\Psi$ is a $2$-symplectic invariant form if and only if 
    \begin{equation}\label{2-symplectic and SKT on 6-dimensional nilmanifolds}
        \begin{cases}
            |A|^{2} + |D|^{2} + |E|^{2} + 2 \text{Re}(\overline{B}C) = 0, \\
            \frac{1}{2} (\overline{C} - \overline{B}) - N \overline{E} = 0.
        \end{cases}
    \end{equation}
\end{theorem}
\begin{proof}  
    As described above, from \cite{FPS2004}, the diagonal Hermitian metric is SKT if and only if $J$ has a basis of $(1,0)$-forms $\{\alpha^{1}, \alpha^{2}, \alpha^{3}\}$ which satisfies \eqref{Structure equations for 6-dimensional Nilmanifolds} and $|A|^{2} + |D|^{2} + |E|^{2} + 2 \text{Re}(\overline{B}C) = 0$.
    
    Since $\omega = \frac{i}{2} \sum_{j=1}^{3} \alpha^{j \overline{j}}$, then $r,s,t = 1$, while $u,v,w = 0$ (see \eqref{Gauduchon metric}). Furthermore,
    \begin{equation}\label{Diagonal SKT metric on FPS}
        \partial \overline{\partial} \omega = \frac{i}{2} (|A|^{2} + |D|^{2} + |E|^{2} + 2\text{Re}(B \overline{C})) \alpha^{1 \overline{1} 2 \overline{2}} = 0,
    \end{equation}
    and 
    \begin{equation*}
        \eta \doteq \omega^{2} = -\frac{1}{2}(\alpha^{1 \overline{1} 2 \overline{2}} + \alpha^{1 \overline{1} 3 \overline{3}} + \alpha^{2 \overline{2} 3 \overline{3}}).
    \end{equation*}
    Thus, from \eqref{Condition for existence of 2-symplectic forms on 6-dimensional Nilmanifolds}, $\Psi \doteq \lambda^{3,1} + \omega^{2} + \overline{\lambda^{3,1}}$ is a $2$-symplectic form if and only if 
    \begin{equation*}
            \frac{1}{2} (\overline{C} - \overline{B}) - N \overline{E} = 0.
    \end{equation*}
    Hence, the thesis follows.
\end{proof}
\begin{remark}
    If $N=0$ or $E=0$ or $B=0$ or $C=0$, the system \eqref{2-symplectic and SKT on 6-dimensional nilmanifolds} implies that $A=B=C=D=E=0$, thus $M$ is a torus. If $N,E,B,C \neq 0$, we can easily observe that there are non trivial solutions of the system. For instance, let us suppose that $A=D=0$ and $N,E,B,C \neq 0$, then, 
    \begin{equation*}
        B \doteq x + iy, \quad C \doteq u+iv, \quad N \doteq s+it,
    \end{equation*}
    where $x,y,u,v,s,t \in \R$. Thus,
    \begin{equation*}
        \overline{B}C = xu + yv +i(xv -yu),
    \end{equation*}
    and, since $E = \frac{(C-B)}{2\overline{N}}$, the system \eqref{2-symplectic and SKT on 6-dimensional nilmanifolds} is equivalent to 
    \begin{equation*}
        \frac{1}{4(s^{2} + t^{2})} \big((u-x)^{2} + (v-y)^{2} \big) + xu + yv = 0.
    \end{equation*}
    Using the following notation $a \doteq |N|^{2} = s^{2} + t^{2}$, we get 
    \begin{equation*}
        (u-x)^{2} + (v-y)^{2} + 4axu + 4ayv = 0,
    \end{equation*}
    thus, once $u,v$ are fixed, the equation becomes 
    \begin{equation*}
        x^{2} + u^{2} - 2xu + y^{2} + v^{2} -2yv + 4axu + 4ayv = 0.
    \end{equation*}
    So 
    \begin{equation*}
        x^{2} + y^{2} + (4a-2) xu + (4a-2) yv + v^{2} + u^{2} = 0,
    \end{equation*}
    which is the equation for a circle in $\R^{2}$. In order to be sure that the circle exists we must have 
    \begin{equation*}
        (4a-2)^{2}u^{2} + (4a-2)^{2} v^{2} - 4v^{2} - 4u^{2} > 0.
    \end{equation*}
    Hence
    \begin{equation*}
        16a^{2}u^{2} - 16au^{2} + 16a^{2}v^{2} - 16a v^{2} > 0,
    \end{equation*}
    so, 
    \begin{equation*}
        16(u^{2} + v^{2})(a^{2} - a) > 0.
    \end{equation*}
    Finally, if $a>1$ then there is a circle of solutions.
\end{remark}

\subsection{$3$-symplectic form on families of $8$-dimensional nilmanifolds}
Let us consider the family of $8$-dimensional nilmanifolds studied by A. Fino and A. Tomas\-sini in \cite{FT2011}. They use this family to prove that, in general, there is no relationship between SKT and Astheno-K\"ahler metrics. Indeed, they show that, under a particular choiche of coefficients, the diagonal metric can be an Astheno-K\"ahler metric but not a SKT metric.

Let us consider the set of $(1,0)$-forms $\{ \eta^{1}, \eta^{2}, \eta^{3}, \eta^{4}\}$, such that 
\begin{equation}\label{Structure equations for 8-dimensional Nilmanifolds with SKT and Astheno}
    \begin{cases}
        d \eta^{j} = 0, \quad \forall j = 1,2,3, \\
        d \eta^{4} = a_{1} \, \eta^{1} \wedge \eta^{2}  + a_{2} \, \eta^{1} \wedge \eta^{3}  + a_{3} \, \eta^{1} \wedge \eta^{\overline{1}}  + a_{4} \, \eta^{1} \wedge \eta^{\overline{2}}  \\
        \quad \quad \quad + a_{5} \, \eta^{1} \wedge \eta^{\overline{3}}  + a_{6} \, \eta^{2} \wedge \eta^{3} + a_{7} \, \eta^{2} \wedge \eta^{\overline{1}}  + a_{8} \, \eta^{2} \wedge \eta^{\overline{2}} \\
        \quad \quad \quad + a_{9} \, \eta^{2} \wedge \eta^{\overline{3}}  + a_{10} \, \eta^{3} \wedge \eta^{\overline{1}}  + a_{11} \, \eta^{3} \wedge \eta^{\overline{2}}  + a_{12} \, \eta^{3} \wedge \eta^{\overline{3}}, 
    \end{cases}
\end{equation}
where $a_{j} \in \C$, for $j=1, \dots 12$. 

The complex forms $\{\eta^{1}, \dots, \eta^{4}\}$ span the dual of a Lie algebra $\mathfrak{g}$, which depends on the coefficients $a_{1}, \dots , a_{12}$ and it is $2$-step nilpotent. Furthermore, the almost complex structure $J$ defined by \eqref{Structure equations for 8-dimensional Nilmanifolds with SKT and Astheno} is integrable. Let us denote by $G$ the simply connected Lie group with Lie algebra $\mathfrak{g}$. Due to Malcev's Theorem (\cite{M1951}) and the nilpotency of $G$, if $a_{1},\dots, a_{12} \in \Q[i]$, then, there exists a lattice $\Gamma$ such that $M \doteq \Gamma \backslash G$ is a nilmanifold.

Let us study the existence of $3$-symplectic forms on $M$.
\begin{theorem}\label{3-symplectic forms on a family of 8-dimensional nilmanifolds}
    Let $(M \doteq \Gamma \backslash G,J,g)$ be the $8$-dimensional nilmanifold equipped with the invariant complex structure $J$ that stems from \eqref{Structure equations for 8-dimensional Nilmanifolds with SKT and Astheno} and the diagonal metric $g$, whose fundamental form is denoted by $\omega$. Let $\Psi$ be an invariant, real, transverse $6$-form, i.e., $\Psi \doteq \lambda^{4,2} + \eta + \overline{\lambda^{4,2}}$, where 
    \begin{equation*}
        \begin{cases}
            & \lambda^{4,2} \doteq L_{1} \eta^{1234 \overline{12}} + L_{2} \eta^{1234 \overline{13}} + L_{3} \eta^{1234 \overline{14}} \\
            & \quad \quad \quad + M_{1} \eta^{1234 \overline{23}} + M_{2} \eta^{1234 \overline{24}} + N \eta^{1234 \overline{34}}, \\
            & \eta \doteq \omega^{3},
        \end{cases}
    \end{equation*}
    and $L_{1},L_{2},L_{3},M_{1},M_{2},N \in \C$. Then $\Psi$ is a $3$-symplectic form if and only if 
    \begin{equation}\label{Condition for existence of 3-symplectic forms on 8-dimensional Nilmanifolds}
        \frac{3}{4}i(a_{3} + a_{8} + a_{12}) - \overline{L_{3}} a_{6} + \overline{M_{2}} a_{2} - \overline{N} a_{1} = 0. 
    \end{equation}
\end{theorem}
\begin{proof}
    Since 
        \begin{equation*}
            \eta = \frac{3}{4} i(\eta^{123 \overline{123}} + \eta^{124 \overline{124}} + \eta^{134 \overline{134}} + \eta^{234 \overline{234}}),
        \end{equation*}
    and, from \eqref{Structure equations for 8-dimensional Nilmanifolds with SKT and Astheno}, we get 
    \begin{equation*}
        \begin{split}
            & d \eta^{124 \overline{124}} = a_{12} \eta^{123 \overline{1234}} - \overline{a_{12}} \eta^{1234 \overline{123}}, \quad d \eta^{134 \overline{134}} = a_{8} \eta^{123 \overline{1234}} - \overline{a_{8}} \eta^{1234 \overline{123}}, \\
            &  d \eta^{234 \overline{234}} = a_{3} \eta^{123 \overline{1234}} - \overline{a_{3}} \eta^{1234 \overline{123}}, \quad d \eta^{123 \overline{123}} = 0,
        \end{split}
    \end{equation*}
    then
    \begin{equation*}
        d \eta =\frac{3}{4} i (a_{3} + a_{8} + a_{12}) \eta^{123 \overline{1234}} - \frac{3}{4} i (\overline{a_{3} + a_{8} + a_{12}}) \eta^{1234 \overline{123}}.
    \end{equation*}
    Furthermore,
    \begin{equation*}
        d \lambda^{4,2} = (M_{2} \overline{a_{2}} - L_{3}\overline{a_{6}} - N \overline{a_{1}}) \eta^{1234 \overline{123}}, \quad d \overline{\lambda^{4,2}} = (\overline{M_{2}} a_{2} - \overline{L_{3}} a_{6} - \overline{N} a_{1}) \eta^{123 \overline{1234}}.
    \end{equation*}
    Thus, the thesis follows.
\end{proof}
In \cite{FT2011}, the authors proved that the diagonal metric $g$ is Astheno-K\"ahler if and only if
\begin{equation*}
    \begin{split}
        & |a_{1}|^{2} + |a_{2}|^{2} + |a_{4}|^{2} + |a_{5}|^{2} + |a_{6}|^{2} + |a_{7}|^{2} + |a_{9}|^{2} \\
        & + |a_{10}|^{2} + |a_{11}|^{2} = 2 \text{Re} \big(a_{3} \overline{a_{8}} + a_{3} \overline{a_{12}} + a_{8} \overline{a_{12}} \big).
    \end{split}
\end{equation*}
Notably, for a specific choice of coefficients, i.e., $a_{8}=0$, $|a_{4}|^{2} + |a_{11}|^{2} \neq 0$, the metric $g$ is Astheno-K\"ahler but not SKT. Moreover, if $a_{8}=0$, then $g$ is both SKT and Astheno-K\"ahler if $a_{1}=a_{4}=a_{6}=a_{7}=a_{9}=a_{11}=0$.

In this case, a $3$-symplectic form can coexist with a metric that is both SKT and Astheno-K\"ahler. As a consequence we get the following 

\begin{theorem}\label{3-symplectic-astheno-skt}
    Let $(M \doteq \Gamma \backslash G,J,g)$ be as in Theorem \ref{3-symplectic forms on a family of 8-dimensional nilmanifolds}. Suppose that $a_{8}=0$. Then the diagonal metric $g$ is both SKT, Astheno-K\"ahler and there exists a $3$-symplectic form $\Psi$, as aforementioned, if and only if 
    \begin{equation}\label{3-Symplectic plus Astheno-Kahler and SKT}
        \begin{cases}
            a_{1}=a_{4}=a_{6}=a_{7}=a_{9}=a_{11}=0, \\
            |a_{2}|^{2} + |a_{5}|^{2} + |a_{10}|^{2} = 2 \text{Re}(a_{3} \overline{a_{12}}), \\
            \frac{3}{4} i (a_{3} + a_{12}) + \overline{M_{2}} a_{2} = 0.
        \end{cases}
    \end{equation}
\end{theorem}
\begin{proof}
    In order to prove the theorem, we have to combine \eqref{Condition for existence of 3-symplectic forms on 8-dimensional Nilmanifolds} and \cite[Theorem 2.7]{FT2011}.    
\end{proof}
\begin{remark}
    On such manifolds we can prove that if the diagonal metric is not Astheno-K\"ahler, then there cannot exists a $1$-pluriclosed form. Indeed, if $\omega$ denotes the fundamental form of the diagonal Hermitian metric, then 
    \begin{equation*}
        \begin{split}
            \partial \overline{\partial} \omega^{2} & = -\frac{1}{2} \partial \overline{\partial} (\sum_{j<l} \eta^{j \overline{j} l \overline{l}}) \\
            & = \frac{1}{2}(|a_{1}|^{2} + |a_{2}|^{2} + |a_{4}|^{2} + |a_{5}|^{2} + |a_{6}|^{2} + |a_{7}|^{2} + |a_{9}|^{2} \\
            & + |a_{10}|^{2} + |a_{11}|^{2} - 2 \text{Re} \big(a_{3} \overline{a_{8}} + a_{3} \overline{a_{12}} + a_{8} \overline{a_{12}} \big)) \eta^{123 \overline{123}}.
        \end{split}
    \end{equation*}
    Thus, we can use Lemma \ref{No-existence of p-pluriclosed} to conclude that if the diagonal metric is not Astheno-K\"ahler, then we cannot have a SKT metric compatible with the same complex structure. 
\end{remark}

\begin{remark}
    From system \eqref{3-Symplectic plus Astheno-Kahler and SKT}, we can easily observe that, if $a_{2}=0$ or $a_{3}=0$ or $a_{12}=0$ or $M_{2}=0$, then the complex manifold $M$ must be the torus.
    While, if $a_{2},a_{3},a_{12},M_{2} \neq 0$, there are non trivial solutions of system \eqref{3-Symplectic plus Astheno-Kahler and SKT}. Indeed, if we suppose that $a_{5}=a_{10}=0$ and $a_{3}=a_{12}$, then system \eqref{3-Symplectic plus Astheno-Kahler and SKT} becomes
    \begin{equation*}
        \begin{cases}
            |a_{2}|^{2} = 2|a_{3}|^{2}, \\
            a_{2} = \frac{3i}{2} \frac{a_{3}}{\overline{M}},
        \end{cases}
    \end{equation*}
    thus,
    \begin{equation*}
        \frac{9}{4} |a_{3}|^{2} = 2 |M|^{2} |a_{3}|^{2}.
    \end{equation*}
    Hence, if  
    \begin{equation*}
        |M|^{2} = \frac{9}{8},
    \end{equation*}
    there are non-trivial solutions of the system. 
\end{remark}

\subsection{$4$-symplectic form on families of $10$-dimensional nilmanifolds}
T. Sferruzza and the second author, in \cite{ST2023}, constructed a family of $5$-dimensional nilmanifolds that can admit the coexistence of $g$ and $g^{'}$ Hermitian metrics, both compatible with the complex structure $J$, and such that one is balanced and the other one is Astheno-K\"ahler. They considered the set of complex $(1,0)$-forms $\{\sigma^{1}, \dots, \sigma^{5}\}$ such that 
\begin{equation}\label{Structure equations for 10-dimensional Nilmanifolds with Astheno and 2-pluriclosed}
    \begin{cases}
        d \sigma^{j}=0, \quad \forall j= 1, \dots, 4, \\
        d \sigma^{5} = a_{1} \, \sigma^{12} + a_{2} \, \sigma^{13} + a_{3} \, \sigma^{14} + a_{4} \, \sigma^{1 \overline{1}} + a_{5} \, \sigma^{1 \overline{2}} + a_{6} \, \sigma^{1 \overline{3}} + a_{7} \, \sigma^{1 \overline{4}} \\
        \quad \quad \quad + b_{1} \, \sigma^{23} + b_{2} \, \sigma^{24} + b_{3} \, \sigma^{2 \overline{1}} + b_{4} \, \sigma^{2 \overline{2}} + b_{5} \, \sigma^{2 \overline{3}} + b_{6} \, \sigma^{2 \overline{4}} \\
        \quad \quad \quad + c_{1} \, \sigma^{34} + c_{2} \, \sigma^{3 \overline{1}} + c_{3} \, \sigma^{3 \overline{2}} + c_{4} \, \sigma^{3 \overline{3}} + c_{5} \, \sigma^{3 \overline{4}} \\
        \quad \quad \quad + d_{1} \, \sigma^{4 \overline{1}} + d_{2} \, \sigma^{4 \overline{2}} + d_{3} \, \sigma^{4 \overline{3}} + d_{4} \, \sigma^{4 \overline{4}},
    \end{cases}
\end{equation}
where $a_{r}, b_{s}, c_{t}, d_{u} \in \C[i]$, for $r=1, \dots , 7$, $s = 1, \dots 6$, $t=1 \dots, 5$, $d=1, \dots 4$. Then, the complex forms $\{\sigma^{1}, \dots, \sigma^{5}\}$ span the dual of a Lie algebra $\mathfrak{g}$, which is $2$-step nilpotent and depends on the choice of the coefficients. Furthermore, the almost complex structure $J$ defined by \eqref{Structure equations for 10-dimensional Nilmanifolds with Astheno and 2-pluriclosed} is integrable. Let us denote by $G$ the simply connected Lie group with Lie algebra $\mathfrak{g}$. Due to Malcev's Theorem (\cite{M1951}) and the nilpotency of $G$, if $a_{r}, b_{s}, c_{t}, d_{u} \in \Q[i]$, for $r=1, \dots , 7$, $s = 1, \dots 6$, $t=1 \dots, 5$, $d=1, \dots 4$, then there exists a lattice $\Gamma$ such that $M \doteq \Gamma \backslash G$ is a nilmanifold.
Let us study the existence of $4$-symplectic forms on $M$.
\begin{theorem}\label{4-symplectic forms on a family of 10-dimensional nilmanifolds}
    Let $(M \doteq \Gamma \backslash G,J,g)$ be the $10$-dimensional nilmanifold equipped with the invariant complex structure $J$ that stems from \eqref{Structure equations for 10-dimensional Nilmanifolds with Astheno and 2-pluriclosed} and the diagonal metric $g$, whose fundamental form is denoted by $\omega$. Let $\Psi$ be an invariant, real, transverse $8$-form, i.e., $\Psi \doteq \lambda^{5,3} + \eta + \overline{\lambda^{5,3}}$, where 
    \begin{equation*}
        \begin{cases}
             & \lambda^{5,3} \doteq L_{1} \sigma^{12345 \overline{123}} + L_{2} \sigma^{12345 \overline{124}} + L_{3} \sigma^{12345 \overline{125}} + M_{1} \sigma^{12345 \overline{134}} + M_{2} \sigma^{12345 \overline{135}} \\ 
             &\quad \quad \quad + N_{1} \sigma^{12345 \overline{145}} + S_{1} \sigma^{12345 \overline{234}}  + S_{2} \sigma^{12345 \overline{235}} + S_{3} \sigma^{12345 \overline{245}} + P \sigma^{12345 \overline{345}}, \\
             & \eta \doteq \omega^{4},
        \end{cases}
    \end{equation*}
    and $L_{1}, L_{2}, L_{3}, M_{1}, M_{2}, N_{1}, S_{1}, S_{2}, S_{3}, P \in \C$. Then, $\Psi$ is a $4$-symplectic form if and only if 
    \begin{equation}\label{Condition for existence of 4-symplectic forms on 10-dimensional Nilmanifolds}
         \frac{3}{2}(d_{4} + c_{4} + b_{4} + a_{4}) - \overline{L_{3}} c_{1} + \overline{M_{2}} b_{2} - \overline{N_{1}} b_{1} - \overline{S_{2}} a_{3} +\overline{S_{3}} a_{2} - \overline{P} a_{1} = 0.  
    \end{equation}
\end{theorem}
\begin{proof}
    Since 
    \begin{equation*}
        \eta = \frac{3}{2} \big( \sigma^{1234 \overline{1234}} + \sigma^{1235 \overline{1235}} + \sigma^{1245 \overline{1245}} + \sigma^{1345 \overline{1345}} + \sigma^{2345 \overline{2345}} \big),
    \end{equation*}
    then, by a straightforward computation
    \begin{equation*}
        d \eta = \frac{3}{2} \big(a_{4} + b_{4} + c_{4} + d_{4} \big) \eta^{1234 \overline{12345}} + \frac{3}{2} \big(\overline{a_{4}} + \overline{b_{4}} + \overline{c_{4}} + \overline{d_{4}} \big) \eta^{12345 \overline{1234}}.
    \end{equation*}
    Furthermore,
    \begin{equation*}
        \begin{split}
        & d \lambda^{5,3} = \big(-L_{3} \overline{c_{1}} + M_{2}\overline{b_{2}} - N_{1} \overline{b_{1}} - S_{2}\overline{a_{3}} + S_{3}\overline{a_{2}} - P\overline{a_{1}} \big) \sigma^{12345 \overline{1234}}, \\
        & d \overline{\lambda^{5,3}} = \big(-c_{1} \overline{L_{3}} + b_{2}\overline{M_{2}} - b_{1} \overline{N_{1}} - a_{3}\overline{S_{2}} + a_{2}\overline{S_{3}} - a_{1}\overline{P} \big)\sigma^{1234 \overline{12345}}.
        \end{split}
    \end{equation*}
    Thus, $d \Psi = 0$ if and only if 
    \begin{equation*}
        \frac{3}{2} \big(a_{4} + b_{4} + c_{4} + d_{4} \big) -c_{1} \overline{L_{3}} + b_{2}\overline{M_{2}} - b_{1} \overline{N_{1}} - a_{3}\overline{S_{2}} + a_{2}\overline{S_{3}} - a_{1}\overline{P} = 0.
    \end{equation*}  
\end{proof}

Let $g$ denotes the diagonal metric on $M$, and $\omega$ the fundamental form of $g$. In \cite{ST2023}, the authors proved that the diagonal metric $g$ is Astheno-K\"ahler if and only if 
\begin{equation*}
        \begin{split}
            2 \text{Re} \big(d_{4} \overline{a_{4}} + d_{4} \overline{b_{4}} + d_{4} \overline{c_{4}} + c_{4}\overline{ a_{4}} + c_{4} \overline{b_{4}} + b_{4} \overline{a_{4}}  \big) = & \, |a_{1}|^{2} + |a_{2}|^{2} + |a_{3}|^{2} + |a_{5}|^{2} \\
            & + |a_{6}|^{2} + |a_{7}|^{2} + |b_{1}|^{2} + |b_{2}|^{2} \\
            & +|b_{3}|^{2} + |b_{5}|^{2} + |b_{6}|^{2} + |c_{1}|^{2} \\
            & + |c_{2}|^{2} + |c_{3}|^{2} + |d_{1}|^{2} + |d_{2}|^{2}.
        \end{split}
    \end{equation*}
Moreover, they proved that, if 
\begin{equation}\label{Astheno plus 2-pluriclosed}
    \begin{split}
        & a_{2} = a_{3} = a_{5} = a_{6} = a_{7} = b_{1} = b_{2} = b_{3} = b_{5} = b_{6} = c_{2} = c_{3} = \\
        & = c_{5} = d_{1} = d_{2} = d_{3} = 0,
    \end{split}
\end{equation}
then 
$g$ is Astheno-K\"ahler and $\partial \overline{\partial} \omega^{2} = 0$ if and only if 
\begin{equation*}
    \begin{cases}
         2 \text{Re} \big(d_{4} \overline{a_{4}} + d_{4} \overline{b_{4}} + d_{4} \overline{c_{4}} \big) = |c_{1}|^{2}, \\
         2 \text{Re} \big(c_{4}\overline{ a_{4}} + c_{4} \overline{b_{4}} + b_{4} \overline{a_{4}} \big) = |a_{1}|^{2}, \\
         \text{Re}\big(c_{4} \overline{b_{4}} - d_{4}\overline{a_{4}} \big)=0 , \\
         \text{Re} \big(b_{4}\overline{d_{4}} - c_{4}\overline{a_{4}} \big) = 0.
    \end{cases}
\end{equation*}
In this case, a $4$-symplectic form can coexist with both an Astheno-K\"ahler metric and a $2$-pluriclosed form. As a direct consequence we obtain the following
\begin{theorem}\label{4-symplectic and Astheno and 2-pluriclosed}
    Let $(M \doteq \Gamma \backslash G,J,g)$ be as in Theorem \ref{4-symplectic forms on a family of 10-dimensional nilmanifolds} and suppose that \eqref{Astheno plus 2-pluriclosed} holds. $g$ is Astheno-K\"ahler, $\omega^{2}$ is $2$-pluriclosed and there exists a $4$-symplectic form $\Psi$ as aforementioned if and only if 
    \begin{equation}\label{2-pluriclosed, Astheno-Kahler and 4-symplectic}
       \begin{cases}
         2 \text{Re} \big(d_{4} \overline{a_{4}} + d_{4} \overline{b_{4}} + d_{4} \overline{c_{4}} \big) = |c_{1}|^{2}, \\
         2 \text{Re} \big(c_{4}\overline{ a_{4}} + c_{4} \overline{b_{4}} + b_{4} \overline{a_{4}} \big) = |a_{1}|^{2}, \\
         \text{Re}\big(c_{4} \overline{b_{4}} - d_{4}\overline{a_{4}} \big)=0 , \\
         \text{Re} \big(b_{4}\overline{d_{4}} - c_{4}\overline{a_{4}} \big) = 0, \\
         \frac{3}{2} \big(a_{4} + b_{4} + c_{4} + d_{4} \big) - c_{1} \overline{L_{3}} - a_{1}\overline{P} = 0.
        \end{cases}
    \end{equation}
\end{theorem}
\begin{proof}
    In order to prove the theorem, we have to combine \eqref{Condition for existence of 4-symplectic forms on 10-dimensional Nilmanifolds} and \cite[Theorem 4.1]{ST2023}.
\end{proof}
\subsection{Non K\"ahler manifolds with trivial first Chern class and non vanishing first Bott-Chern class}
Here, we provide an example of a manifold which has first Chern class vanishing but first Bott-Chern class non vanishing. We refer to \cite{T2015} for notation and convention. We recall that, if $M$ is a complex manifold, the Bott-Chern cohomology of $M$ is the bigraded algebra 
\begin{equation*}
    H^{\bullet, \bullet}_{BC}(M) \doteq \frac{\text{ker} \partial \cap \text{ker} \overline{\partial}}{\text{im} \partial \overline{\partial}}.
\end{equation*}
The first Bott-Chern class is given by the map $c_{1}^{BC}:\text{Pic}(M) \to H^{1,1}_{BC}(M;\R)$. As it is showed in \cite{T2015}, it is easy to show that this is a well defined map and that, if $(M,g)$ is a Hermitian manifold with fundamental form denoted by $\omega$, then $c_{1}^{BC}(M) \doteq c_{1}^{BC}(K^{\ast}_{M})$, where $K_{M}$ is the canonical bundle, is represented locally by 
\begin{equation*}
    \text{Ric}(\omega) = -i \partial \overline{\partial} \, \text{log} \, \text{det} g.
\end{equation*}
Non-K\"ahler Calabi Yau manifolds are compact complex Hermitian manifolds with $c_{1}^{BC}(M)$ vanishing in $H^{1,1}_{BC}(M;\R)$. These manifolds are strictly contained in the class of manifolds that have $c_{1}(M)$ vanishing in $H^{2}(M;\R)$. V. Tosatti non-K\"ahler Calabi-Yau manifolds providing several examples. In particular, he showed that the the family of small deformation of Nakamura manifold \cite[Example III-(3b)]{N1975} (see \cite[Proposition 1 (case 2.)]{AK2017}), which do not satisfy the $\partial \overline{\partial}$-Lemma, have vanishing first Bott-Chern class. Furthermore, he showed that Hopf manifolds have first Chern class vanishing but non vanishing first Bott-Chern class. 

Let us consider the Lie group $S_{1, \frac{\pi}{2}}$ given by the following semidirect product 
\begin{equation*}
    S_{1, \frac{\pi}{2}} \doteq \R \ltimes_{\varphi}(\R \times \R^{2} \times \R^{2}),
\end{equation*}
where the action of $\varphi$ is given by 
\begin{equation*}
    \begin{pmatrix}
        t \\
        x_{1} \\
        x_{2} \\
        x_{3} \\
        x_{4} \\
        x_{5}
   \end{pmatrix} \cdot
    \begin{pmatrix}
        t^{'} \\
        x_{1}^{'} \\
        x_{2}^{'} \\
        x_{3}^{'} \\
        x_{4}^{'} \\
        x_{5}^{'}
   \end{pmatrix} = 
   \begin{pmatrix}
        t^{'} + t \\
        e^{-t}x_{1}^{'} + x_{1} \\
        e^{\frac{t}{2}}x_{2}^{'} + x_{2} \\
        e^{\frac{t}{2}}x_{3}^{'} + x_{3} \\
        x_{4}^{'} \text{cos}(\frac{\pi t}{2}) - x_{5}^{'} \text{sin}(\frac{\pi t}{2}) + x_{4} \\
        x_{4}^{'} \text{sin}(\frac{\pi t}{2}) + x_{5}^{'} \text{cos}(\frac{\pi t}{2}) + x_{5},
   \end{pmatrix} 
\end{equation*}
and $(t,x_{1},x_{2},x_{3},x_{4},x_{5})$ are global coordinates on $\R^{6}$. The following left-invariant $1$-forms
\begin{equation*}
    \begin{split}
        e^{1} \doteq e^{t} d x_{1}, \quad e^{2} \doteq dt, \quad e^{3} \doteq e^{- \frac{t}{2}} dx_{2}, \quad e^{4} \doteq e^{- \frac{t}{2}} dx_{3}\\
        e^{5} \doteq \text{cos}(\frac{\pi t}{2})dx_{4} + \text{sin}(\frac{\pi t}{2})dx_{5}, \quad e^{6} \doteq - \text{sin}(\frac{\pi t}{2})dx_{4} + \text{cos}(\frac{\pi t}{2})dx_{5},
    \end{split}
\end{equation*}
satisfies 
\begin{equation*}
    \begin{cases}
        d e^{1} = -e^{12}, \\
        d e^{2} = 0, \\
        d e^{3} = -\frac{1}{2}e^{23}, \\
        d e^{4} = -\frac{1}{2}e^{24}, \\
        d e^{5} = \frac{\pi}{2}e^{26}, \\
        d e^{6} = - \frac{\pi}{2} e^{25},
    \end{cases}
\end{equation*}
hence, $S_{1, \frac{\pi}{2}}$ is a $2$-step solvable Lie group.

As it is showed in \cite[Lemma 3.1]{FT2009}, $S_{1, \frac{\pi}{2}}$ admits a compact quotient $M \doteq \Gamma \backslash S_{1,\frac{\pi}{2}}$, where $\Gamma \subseteq S_{1, \frac{\pi}{2}}$ is a lattice. Furthermore, $M$ can be equipped with the following integrable almost complex structure
\begin{equation*}
    \phi^{1} \doteq e^{1} + i e^{2}, \quad \phi^{2} \doteq e^{3} + i e^{4}, \quad \phi^{3} \doteq e^{5} + i e^{6}.
\end{equation*}
Local holomorphic coordinates for $M$ are given by 
\begin{equation*}
    z_{1} \doteq x_{1} - ie^{-t}, \quad z_{2} \doteq x_{2} + i x_{3}, \quad z_{3} \doteq x_{4} + i x_{5},
\end{equation*}
indeed 
\begin{equation*}
    d z_{1} = e^{-t} \phi^{1}, \quad dz_{2} = e^{\frac{1}{2}t}\phi^{2}, \quad dz_{3} = \text{cos}(\frac{\pi t}{2}) \phi^{3} + \text{sin}(\frac{\pi t}{2}) \phi^{3}.
\end{equation*}
We want to show that the solvable manifold $M$, equipped with the invariant complex structure aforementioned, has $c_{1}(M)=0$, but $c_{1}^{BC}(M) \neq 0$. 

Let us consider 
\begin{equation*}
    \omega \doteq \frac{i}{2} \big(\phi^{1 \overline{1}} + \phi^{2 \overline{2}} + \phi^{3 \overline{3}} \big);
\end{equation*}
then $\omega$ is the fundamental form of a Hermitian metric on $M$ and it turns out that $\omega$ is SKT. It holds
\begin{equation*}
    \begin{split}
        \omega^{3} & = \frac{3}{2} i^{3} \phi^{123 \overline{123}} = \widetilde{c}_{3} e^{123456} = \widetilde{c}_{3} d x_{1} \wedge dt \wedge d x_{2} \wedge d x_{3} \wedge d x_{4} \wedge d x_{5} \\
        & = \widetilde{c}_{3} e^{t} dz_{1} \wedge d \overline{z_{1}} \wedge d z_{2} \wedge d \overline{z_{2}} \wedge d z_{3} \wedge d \overline{z_{3}},
    \end{split}
\end{equation*}
where $\widetilde{c}_{3} \in \R$, hence, the Ricci form is given, locally, by
\begin{equation*}
    \text{Ric}(\omega) = -i \partial \overline{\partial} \, \text{log}(\,\widetilde{c}_{3} \, e^{t}) = -i \partial \overline{\partial} t.
\end{equation*}
Since 
\begin{equation*}
    d t = e^{t}\bigg(\frac{d z_{1} - d \overline{z_{1}}}{2i}\bigg), \quad de^{t} = e^{2t} \bigg(\frac{dz_{1} - d\overline{z_{1}}}{2i} \bigg),   
\end{equation*}
then 
\begin{equation*}
    \text{Ric}(\omega) = -\frac{i}{4} e^{2t} dz_{1} \wedge d \overline{z_{1}}.
\end{equation*}
We can easily see that it is $d$-exact, indeed 
\begin{equation*}
    d \phi^{1} = e^{t} dt \wedge dx_{1} = e^{2t} \bigg( \frac{dz_{1} - d \overline{z_{1}}}{2i} \bigg) \wedge \bigg(\frac{dz_{1} + d \overline{z_{1}}}{2} \bigg) = - \frac{i}{2} e^{2t}  dz_{1} \wedge d \overline{z_{1}},
\end{equation*}
hence $\frac{1}{2} d\phi^{1} = \text{Ric}(\omega)$.
Meanwhile, $\text{Ric}(\omega)$ is semidefinite negative, but it is not identically zero. Hence, arguing as in \cite[Example 3.3]{T2015}, if there exists a function such that $\text{Ric}(\omega) = i \partial \overline{\partial} F$, then $F$ must be constant and $\text{Ric}(\omega)$ must be zero. This yields the thesis.

Furthermore, since 
\begin{equation*}
    d \phi^{1} = -\frac{i}{2} \phi^{1 \overline{1}},
\end{equation*}
we can conclude from Lemma \ref{No-existence of p-symplectic} that there are no $2$-symplectic forms on $M$.

\nocite{*}
\bibliographystyle{siam}
\bibliography{./reference}

\end{document}